\documentclass[11pt]{amsart}
\usepackage{float}
\usepackage{bbm}
\usepackage{amsfonts}
\usepackage{amsmath}
\usepackage{amssymb}
\usepackage{graphicx}
\usepackage[bindingoffset=0.2in,%
left=1.1in,right=1.1in,top=1.2in,bottom=1.375in,%
footskip=.25in]{geometry}

\usepackage{mathtools}
\usepackage{amsthm}
\usepackage{latexsym}
\usepackage{fancyhdr}
\usepackage{array}
\usepackage{amscd}
\usepackage{lscape}
\usepackage{tikz}
\usepackage{tikz-cd}
\usepackage{float}
\usepackage{bm}
\usepackage{subfigure}
\usepackage{grffile}
\usepackage{lmodern}
\usepackage{array}
\usepackage{longtable}
\usepackage{booktabs}
\usepackage{adjustbox}
\usepackage{diagbox}
\usepackage{algorithm}
 \usepackage{algpseudocode}
 \usepackage{algorithmicx}
 \algdef{SE}[DOWHILE]{Do}{doWhile}{\algorithmicdo}[1]{\algorithmicwhile\ #1}%

\algrenewcommand\algorithmicrequire{\textbf{Input:}}
\algrenewcommand\algorithmicensure{\textbf{Output:}}

\usepackage{ dsfont }
\usepackage{lipsum}
\newcolumntype{C}[1]{>{\centering\arraybackslash}p{#1}}

\definecolor{navy}{HTML}{2F729C} 
\usepackage[hyperfootnotes=false, colorlinks, linkcolor={blue}, citecolor={magenta}, filecolor={blue}, urlcolor={blue}, plainpages=false, pdfpagelabels]{hyperref}
\usepackage[tableposition=above]{caption}
\newcommand{\Q}{{\mathbb Q}}

\newcommand{\GL}{{\rm GL}}

\newcommand{\Gal}{{\rm Gal}}

\newtheorem{theorem}{Theorem}[section]
\newtheorem{lemma}[theorem]{Lemma}

\newtheorem{corollary}[theorem]{Corollary}

  \DeclareFontFamily{U}{wncy}{}
    \DeclareFontShape{U}{wncy}{m}{n}{<->wncyr10}{}
    \DeclareSymbolFont{mcy}{U}{wncy}{m}{n}
    \DeclareMathSymbol{\Sha}{\mathord}{mcy}{"58}

\theoremstyle{definition}

\numberwithin{equation}{section}

\begin{document}

\title[On The Birch and Swinnerton-Dyer formula modulo squares]{On The Birch and Swinnerton-Dyer formula modulo squares for certain quadratic twists of elliptic curves}

\author{Alexander J. Barrios}
\address{Department of Mathematics, University of St. Thomas, St. Paul, MN 55105 USA}
\email{abarrios@stthomas.edu}

\author{Chung Pang Mok}
\address{Shanghai Institute for Mathematics and Interdisciplinary Sciences, Block A, International Innovation Plaza, No. 657 Songhu Road, Yangpu District, Shanghai, China}
\email{cpmok@simis.cn}


\begin{abstract}
Let $E/\mathbb{Q}$ be an elliptic curve with conductor $N=N_+N_-$, where $N_+$ and $N_-$ are coprime and $N_-$ is squarefree. Let $D$ be a positive fundamental discriminant satisfying the modified Heegner hypothesis with respect to $(N_+,N_-)$: primes dividing $N_+$ (resp. $N_-$) split (resp. are inert) in $\Q(\sqrt{D})$; we denote by $E^D/\mathbb{Q}$ the quadratic twist of $E/\mathbb{Q}$ by $D$. In the first half of the paper we consider the situation where $N_-$ is a squarefree product of an odd number of distinct primes, and we show the following: assuming that $E/\mathbb{Q}$ is of analytic rank zero (resp. one), and that the Birch and Swinnerton-Dyer formula holds for $E/\mathbb{Q}$ modulo $(\mathbb{Q}^{\times})^2$, then for those $D$ such that $E^D/\mathbb{Q}$ is of analytic rank one (resp. zero), we also have the validity of the Birch and Swinnerton-Dyer formula for $E^D/\mathbb{Q}$ modulo $(\mathbb{Q}^{\times})^2$. To show this, we establish auxiliary results without rank assumptions. The most difficult case is when $D$ is even, and our proof crucially relies on the recent classification of how local Tamagawa numbers change under quadratic twists. In the final part of the paper analogous results are also obtained in the other situation when $N_-$ is a squarefree product of an even number distinct primes, concerning the case when both $E/\mathbb{Q}$ and $E^D/\mathbb{Q}$ have analytic rank zero (resp. one). 

\bigskip

As a consequence of our work, we obtain that if $E/\mathbb{Q}$ is semistable with conductor $N$ and whose analytic rank is at most one, then for any positive fundamental discriminant $D$ that is coprime to $N$, such that $E^D/\mathbb{Q}$ again has analytic rank at most one, we have that the Birch and Swinnerton-Dyer formula modulo $(\mathbb{Q}^{\times})^2$ holds for $E/\mathbb{Q}$ if and only if it holds for $E^D/\mathbb{Q}$.
\end{abstract}
\maketitle

\section{Introduction}

Let $E/\mathbb{Q}$ be an elliptic curve defined over the field of rational numbers $\mathbb{Q}$. In this paper we are interested in the question: assuming the Birch and Swinnerton-Dyer formula \cite{BSD} (i.e. the full Birch and Swinnerton-Dyer conjecture) holds for $E/\mathbb{Q}$, then what can be said concerning the Birch and Swinnerton-Dyer formula for the various $E^D/\mathbb{Q}$, the quadratic twists of $E/\mathbb{Q}$ by fundamental discriminants $D$? 

\bigskip

At this level of generality, the problem is certainly a difficult one, and we will consider a more restricted setting in this paper. Our setup is as follows: fix an $E/\mathbb{Q}$ whose conductor $N$ is written in the form $N=N_+ N_-$, with $N_+$ and $N_-$ being relatively prime, and such that $N_-$ is squarefree (so in particular $E/\mathbb{Q}$ has multiplicative reduction at all primes dividing $N_-$). As for $D$, in this paper we consider only {\it positive} fundamental discriminants $D$ that satisfy the {\it modified Heegner hypothesis} with respect to $(N_+,N_-)$: all primes dividing $N_+$ split in $\mathbb{Q}(\sqrt{D})$, and all primes dividing $N_-$ are inert in $\mathbb{Q}(\sqrt{D})$ (in particular $D$ is relatively prime to $N$). The conductor of the quadratic twist $E^D/\mathbb{Q}$ is equal to $N \cdot D^2$.

\bigskip

To state our results, we first recall the complex analytic $L$-function $L(s,E/\mathbb{Q})$ associated to the elliptic curve $E/\mathbb{Q}$, which by the modularity theorem has analytic continuation to an entire function on the complex plane. Denote by $\epsilon(E/\mathbb{Q}) \in \{ \pm 1\}$ the sign of the functional equation for $L(s,E/\mathbb{Q})$; similarly for the quadratic twists $E^D/\mathbb{Q}$ we have the complex analytic $L$-function $L(s,E^D/\mathbb{Q})$ and the corresponding sign of the functional equation $\epsilon(E^D/\mathbb{Q}) \in \{\pm 1\}$. Then for those positive fundamental discriminants $D$ that satisfy the modified Heegner hypothesis with respect to $(N_+,N_-)$ as above, we have that $\epsilon(E^D/\mathbb{Q}) = -\epsilon(E/\mathbb{Q}) $ if $N_-$ is a squarefree product of an {\it odd} number of distinct primes, while $\epsilon(E^D/\mathbb{Q}) = \epsilon(E/\mathbb{Q}) $ if $N_-$ is a squarefree product of an {\it even} number of distinct primes.

\bigskip

In the first half of the paper, we consider the situation where $N_-$ is a squarefree product of an odd number of distinct primes (hence $E/\mathbb{Q}$ has at least one prime of multiplicative reduction).

\bigskip
Thus, if the order of vanishing of $L(s,E/\mathbb{Q})$ at $s=1$ is even (resp. odd), then the order of vanishing of $L(s,E^D/\mathbb{Q})$ at $s=1$ is odd (resp. even). We are interested in the case where the order of vanishing of $L(s,E/\mathbb{Q})$ at $s=1$ is zero and the order of vanishing of $L(s,E^D/\mathbb{Q})$ at $s=1$ is one (in which case $\epsilon(E/\mathbb{Q})=+1,\epsilon(E^D/\mathbb{Q})=-1$), respectively the case where the order of vanishing of $L(s,E/\mathbb{Q})$ at $s=1$ is one and the order of vanishing of $L(s,E^D/\mathbb{Q})$ at $s=1$ is zero (in which case $\epsilon(E/\mathbb{Q})=-1, \epsilon(E^D/\mathbb{Q})=+1$); the first case is equivalent to requiring that $L(1,E/\mathbb{Q}) \neq 0$, $L^{\prime}(1,E^D/\mathbb{Q}) \neq 0$, while the second case is equivalent to requiring that $L^{\prime}(1,E/\mathbb{Q}) \neq 0$, $L(1,E^D/\mathbb{Q}) \neq 0$.

\bigskip
By Friedberg-Hoffstein \cite{FH}, applied to the weight two cuspidal newform $f_E$ corresponding to $E/\mathbb{Q}$, we have in any case for $E/\mathbb{Q}$ with $\epsilon(E/\mathbb{Q}) =+1$ (resp. $-1$), that there exists infinitely many positive fundamental discriminants $D$ satisfying the modified Heegner hypothesis with respect to $(N_+,N_-)$ as above, such that $L^{\prime}(1,E^D/\mathbb{Q}) \neq 0$ (resp. $L(1,E^D/\mathbb{Q}) \neq 0$); remark that in the case when $\epsilon(E/\mathbb{Q})=-1$ then we can also use Murty-Murty, specifically Chapter 6 of \cite{MM}.

\bigskip

Here below, the order of vanishing of $L(s,E/\mathbb{Q})$ at $s=1$ will be referred to as the analytic rank of $E/\mathbb{Q}$ (similarly for $E^D/\mathbb{Q}$). The weak form of the Birch and Swinnerton-Dyer conjecture states that the rank of the Mordell-Weil group $E(\mathbb{Q})$ is equal to the analytic rank of $E/\mathbb{Q}$ for any elliptic curve $E/\mathbb{Q}$, and by the works of Gross-Zagier \cite{GrossZagier} and Kolyvagin \cite{Kolyvagin1}, the weak form of the Birch and Swinnerton-Dyer conjecture is known for any $E/\mathbb{Q}$ whose analytic rank is at most one. The full Birch and Swinnerton-Dyer conjecture, that we refer to as the Birch and Swinnerton-Dyer formula, gives a conjectural expression for the leading Taylor coefficient of $L(s,E/\mathbb{Q})$ at $s=1$, in terms of various arithmetic invariants of $E/\mathbb{Q}$: the regulator $\mathrm{Reg}(E/\mathbb{Q})$ of $E/\mathbb{Q}$, the order of the Shafarevich-Tate group $\Sha(E/\mathbb{Q})$ (conjectured to be finite, in which case it has to be a square by the Cassels' pairing), the order of the torsion subgroup of $E(\mathbb{Q})$, the local Tamagawa numbers $c_l(E/\mathbb{Q})$ at primes $l$ dividing the conductor $N$ of $E/\mathbb{Q}$ (which are positive integers and depends only on $E/\mathbb{Q}_l$, where $\mathbb{Q}_l$ is the field of $l$-adic numbers), and the real N\'{e}ron period $\Omega_{E/\mathbb{Q}}^+$ of $E/\mathbb{Q}$:

\bigskip
\[
\frac{1}{r!} L^{(r)}(1,E/\mathbb{Q}) = \frac{\# \Sha(E/\mathbb{Q}) }{(\# E(\mathbb{Q})_{\mathrm{tors}})^2}    \cdot \mathrm{Reg}(E/\mathbb{Q})\cdot  \prod_{l|N} c_l(E/\mathbb{Q}) \cdot \Omega_{E/\mathbb{Q}}^+,
\]
with $r$ being the analytic rank of $E/\mathbb{Q}$ (hence is equal to the rank of $E(\mathbb{Q})$ under the weak form of the conjecture).  

\bigskip
Recall also that by the works of Kolyvagin \cite{K} the order of $\Sha(E/\mathbb{Q})$ is known to be finite (and hence a square) when the analytic rank of $E/\mathbb{Q}$ is at most one.

\bigskip
We are interested in the validity of the Birch and Swinnerton-Dyer formula modulo multiplication by square of non-zero rational numbers (and we refer to this as saying that the formula holds modulo square of rational numbers). In this paper, we show the following result:

\bigskip
\begin{theorem}\label{1.1}
Fix $E/\mathbb{Q}$ whose conductor $N$ is written in the form $N=N_+N_-$ as above, with $N_-$ being a squarefree product of an odd number of distinct primes. Suppose that the analytic rank of $E/\mathbb{Q}$ is equal to zero (resp. one). Given a positive fundamental discriminant $D$ satisfying the modified Heegner hypothesis with respect to $(N_+,N_-)$, such that $E^D/\mathbb{Q}$ is of analytic rank one (resp. zero), we have that the Birch and Swinnerton-Dyer formula modulo square of rational numbers holds for $E^D/\mathbb{Q}$ if and only if it holds for $E/\mathbb{Q}$.
\end{theorem}

\bigskip

Now we discuss the ingredients used for the proof of Theorem~\ref{1.1}. The first is a Gross-Zagier type formula modulo square of nonzero rational numbers. As before, $E/\mathbb{Q}$ is an elliptic curve whose conductor $N$ is written in the form $N=N_+N_-$, with $N_-$ a squarefree product of an odd number of distinct primes. Given a positive fundamental discriminant $D$ that satisfies the modified Heegner hypothesis with respect to $(N_+,N_-)$, denote by $F$ the real quadratic field $\mathbb{Q}(\sqrt{D})$, by $\mathcal{O}_F$ its ring of integers, and by $E/F$ the base change of $E/\mathbb{Q}$ to $F$. For primes $q|N_-$, we have that $q$ is inert in $F$, and we denote by $c_q(E/F)$ the local Tamagawa number of $E/F$ at the prime $q \mathcal{O}_F$. Finally The $L$-function of $E/F$ has the factorization:
\[
L(s,E/F) = L(s,E/\mathbb{Q}) \cdot L(s,E^D,\mathbb{Q})
\]

\bigskip

\begin{theorem}\label{1.2}
With notations as above, we have:
\begin{eqnarray}\label{e1.1}
& & L^{\prime}(1,E/F) \\ &=& \frac{1}{\sqrt{D}}  \prod_{q|N_-} c_q(E/F) \cdot \mathrm{height}_F(\mathbf{P}) \cdot (\Omega^+_{E/\mathbb{Q}})^2 \times \mbox{square of a nonzero rational number}, \nonumber
\end{eqnarray}
where $\mathbf{P} \in E(F) \otimes_{\mathbb{Z}} \mathbb{Q}$ and $\mathrm{height}_F(\mathbf{P})$ is the N\'{e}ron-Tate height function (for $E/F$) evaluated at $\mathbf{P}$. Furthermore $\mathbf{P}$ can be taken to be an element of $E(\mathbb{Q}) \otimes_{\mathbb{Z}} \mathbb{Q}$ if $L(1,E^D/\mathbb{Q}) \neq 0$ (in which case the $\mathbb{Q}$-vector space $E(\mathbb{Q}) \otimes_{\mathbb{Z}} \mathbb{Q}$ is one dimensional if in addition we have $L^{\prime}(1,E/\mathbb{Q}) \neq 0$), and an element of $E^D(\mathbb{Q}) \otimes_{\mathbb{Z}} \mathbb{Q}$ if we have $L(1,E/\mathbb{Q}) \neq 0$ (in which case the $\mathbb{Q}$-vector space $E^D(\mathbb{Q}) \otimes_{\mathbb{Z}} \mathbb{Q}$ is one dimensional if in addition we have $L^{\prime}(1,E^D/\mathbb{Q}) \neq 0$).  
\end{theorem}

\bigskip

Formula \eqref{e1.1} is established in \cite{Mok1}, in the special case when $N_-$ is an odd prime $p$, in the context of a complex analytic Gross-Zagier formula modulo square of rational numbers for Stark-Heegner points. The proof for the general case (namely that $N_-$ is a squarefree product of an odd number of distinct primes) is essentially the same and it is given in the next section (in the more general context of narrow genus class characters of $F$). 

\bigskip

For the second ingredient, we first introduce some further notations. For the moment, we come back to the general setting where $N_-$ is only required to be squarefree and relatively prime to $N_+$. With $E/\mathbb{Q}$ is as before, we have for primes $q|N_-$, that $E/\mathbb{Q}$ has multiplicative reduction (split or non-split) at $q$. We denote by $\widetilde{c}_q(E/\mathbb{Q}) \in \{1,2\}$, determined by the condition that $\widetilde{c}_q(E/\mathbb{Q}) =1$ (resp. $2$) if the $q$-valuation of the minimal discriminant of $E/\mathbb{Q}$ is odd (resp. even).

\bigskip

Next, by the results of \cite{Pal}, we have the following relation between the real N\'{e}ron period $\Omega_{E/\mathbb{Q}}^+$ of $E/\mathbb{Q}$, and the real N\'{e}ron period $\Omega_{E^D/\mathbb{Q}}^+$ of $E^D/\mathbb{Q}$:

\[
\Omega_{E^D/\mathbb{Q}}^+ = \frac{u_D}{\sqrt{D}}  \Omega_{E/\mathbb{Q}}^+,
\]
with $u_D \in \{1,2,4,8\}$ ($u_D =1$ if $D$ is odd). 

\bigskip

The second ingredient for the proof of Theorem~\ref{1.1} is then the following result; here no assumption on the rank of $E/\mathbb{Q}$ or $E^D/\mathbb{Q}$ is needed, and so is of independent interest; in the following for primes $l|D$ the quantity $c_l(E^D/\mathbb{Q})$ is the local Tamagawa number of $E^D/\mathbb{Q}$ at the prime $l$; we have $c_l(E^D/\mathbb{Q}) \in \{1,2,4\}$.

\bigskip
\begin{theorem}\label{1.3}
As before $E/\mathbb{Q}$ is an elliptic curve whose conductor $N$ is written in the form $N=N_+N_-$, with $N_+,N_-$ relatively prime and $N_-$ squarefree. Given a positive fundamental discriminant $D$ that satisfies the modified Heegner hypothesis with respect to $(N_+,N_-)$, we have that the quantity:
\begin{eqnarray}\label{e1.2}
\frac{u_D}{2^{\omega(N_-)}} \cdot \prod_{l|D} c_l(E^D/\mathbb{Q}) \cdot   \prod_{q|N_-} \widetilde{c}_q(E/\mathbb{Q})
\end{eqnarray}
is an even power of $2$. Here $\omega(\cdot)$ is the prime omega function counting the number of distinct prime factors.
\end{theorem}
\bigskip
Theorem~\ref{1.3} is established in \cite{Mok1} in the special case where $N_-$ is an odd prime and $D$ is odd, and it is conjectured there, {\it c.f.} Conjecture 4.1 of {\it loc. cit.} that it holds for even $D$ also (at least in the case where $N_-$ is an odd prime). The case of odd $D$ is much simpler due to the fact that, on the one hand one then has $u_D=1$, and on the other hand, there is an explicit description of the local Tamagawa numbers $c_l(E^D/\mathbb{Q})$ at primes $l|D$ (Proposition 5 of \cite{Rubin07}), and the proof in the case where $D$ is odd essentially follows directly from the quadratic reciprocity law for Jacobi symbols. The case of even $D$ presents much more serious difficulties, and our proof in this paper crucially relies on the recent classification \cite{BRSTTW} of how local Tamagawa numbers change under quadratic twists.

\bigskip

Theorem~\ref{1.1} is then an immediate consequence of Theorem~\ref{1.2} and Theorem~\ref{1.3}, by following the arguments of Section~\ref{Sec4} of \cite{Mok1}, where Theorem~\ref{1.1} is established in {\it loc. cit.} (though not explicitly stated in this way) in the special case where $N_-$ is an odd prime, and $D$ is odd. 

\bigskip

Finally, we deal with the case where $N_-$ is a squarefree product of an even number of distinct primes (in particular, $N_-=1$ is allowed). As we have seen, in this situation we have that $\epsilon(E/\mathbb{Q}) =\epsilon(E^D/\mathbb{Q})$, and the order of vanishing of both $L(s,E/\mathbb{Q})$ and $L(s,E^D/\mathbb{Q})$ at $s=1$ have the same parity. We are interested in the case when the analytic ranks of both $E/\mathbb{Q}$ and $E^D/\mathbb{Q}$ are equal to zero, and also the case when they are both equal to one. We show:

\bigskip

\begin{theorem}\label{1.4}
Fix $E/\mathbb{Q}$ whose conductor $N$ is written in the form $N=N_+N_-$ as above, with $N_-$ being a squarefree product of an even number of distinct primes. Suppose that the analytic rank of $E/\mathbb{Q}$ is equal to zero. Given positive fundamental discriminant $D$ satisfying the modified Heegner hypothesis with respect to $(N_+,N_-)$, such that $E^D/\mathbb{Q}$ is also of analytic rank zero, we have that the Birch and Swinnerton-Dyer formula modulo square of rational numbers holds for $E^D/\mathbb{Q}$ if and only if it holds for $E/\mathbb{Q}$. A similar result is valid concerning the case when both $E/\mathbb{Q}$ and $E^D/\mathbb{Q}$ have analytic rank one, under the additional assumption that $E/\mathbb{Q}$ has at least one prime of multiplicative reduction.
\end{theorem}

\bigskip

The novel feature in the proof of Theorem~\ref{1.4} (given as Theorem~\ref{5.1} and Theorem~\ref{5.2}), at least in the case when $E/\mathbb{Q}$ has at least one prime of multiplicative reduction, is that we have to rely on Theorem~\ref{1.1}, and more precisely its more general version, namely Theorem~\ref{3.4}, which concerns the case where $N_-$ has an odd number of prime factors, to deal with the case where $N_-$ has an even number of prime factors.

\bigskip

From Theorems~\ref{1.1} and~\ref{1.4}, we then obtain the following Corollary~\ref{1.5} (the more general version is stated as Theorem~\ref{5.3}). Here, for this corollary, $E/\mathbb{Q}$ is semistable, but we do not fix a factorization $N=N_+ N_-$ of the conductor $N$ of $E/\mathbb{Q}$, i.e., the choice of $N_+,N_-$ depends on the choice of $D$.

\bigskip

\begin{corollary}\label{1.5}
Let $E/\mathbb{Q}$ be a semistable elliptic curve with conductor $N$, whose analytic rank is at most one. Then for any positive fundamental discriminant $D$ that is relatively prime to $N$, such that $E^D/\mathbb{Q}$ again has analytic rank at most one, we have that the Birch and Swinnerton-Dyer formula modulo square of rational numbers holds for $E/\mathbb{Q}$ if and only if it holds for $E^D/\mathbb{Q}$.
\end{corollary}
\begin{proof}
As $E/\mathbb{Q}$ is semistable, its conductor $N$ is squarefree, and $E/\mathbb{Q}$ has multiplicative reduction at all primes dividing $N$. Given the positive fundamental discriminant $D$ that is relatively prime to $N$, define $N_+$ (resp. $N_-$) as the product of prime divisors of $N$ that split (resp. are inert) in $\mathbb{Q}(\sqrt{D})$. Then $N=N_+ N_-$, with $N_+$ and $N_-$ being relatively prime (and of course $N_-$ is squarefree), and $D$ satisfies the modified Heegner hypothesis with respect to $(N_+,N_-)$. Then the statement of Corollary~\ref{1.5} follows immediately from Theorem~\ref{1.1} and Theorem~\ref{1.4}. 

\bigskip
Specifically, assuming that both the analytic ranks of $E/\mathbb{Q}$ and $E^D/\mathbb{Q}$ are at most one. If $N_-$ is a squarefree product of an odd number of distinct prime factors, then it must be the case that $E/\mathbb{Q}$ has analytic rank zero and $E^D/\mathbb{Q}$ has analytic rank one, or the case that $E/\mathbb{Q}$ has analytic rank one and $E^D/\mathbb{Q}$ has analytic rank zero, and so we apply Theorem~\ref{1.1}.  If $N_-$ is a squarefree product of an even number of distinct prime factors, then it must be the case that $E/\mathbb{Q}$ and $E^D/\mathbb{Q}$ both have analytic rank zero, or the case that $E/\mathbb{Q}$ and $E^D/\mathbb{Q}$ both have analytic rank one, and so we apply Theorem~\ref{1.4} (here we are also using the fact that the conductor of any elliptic curve over $\mathbb{Q}$ is greater than one, and hence there is indeed at least one prime of multiplicative reduction for $E/\mathbb{Q}$, so Theorem~\ref{1.4} in the case when both $E/\mathbb{Q}$ and $E^D/\mathbb{Q}$ have analytic rank one is indeed applicable). 
\end{proof}

\bigskip

Finally, we remark that the methods of this paper could also be adapted to treat the case of quadratic twists by negative fundamental discriminants as well.

\bigskip

The organization of this paper is as follows. In Section~\ref{Sec2}, we establish Theorem~\ref{1.2}. As mentioned above, we will establish this in the more general context of narrow genus class characters of the real quadratic fields $\mathbb{Q}(\sqrt{D})$, namely Theorem~\ref{2.1}. In Section~\ref{Sec3}, we then establish Theorem~\ref{1.1}, modulo Theorem~\ref{1.3}. We again establish a more general version of Theorem~\ref{1.1}, namely Theorem~\ref{3.4} (modulo Theorem~\ref{1.3}). Sections~\ref{Sec4} and~\ref{Sec5} are then independent of the previous sections, and $N_-$ is only required to be squarefree (and relatively prime to $N_+$). We establish Theorem~\ref{1.3} in Section~\ref{Sec4}, and thus finish the proof of Theorem~\ref{1.1} and Theorem~\ref{3.4}. Finally, in Section~\ref{Sec5}, we combine our previous results with additional arguments to establish Theorem~\ref{1.4}; in fact, we establish a more general version of it, given as Theorem~\ref{5.1} and Theorem~\ref{5.2}. We conclude the paper with Theorem~\ref{5.3}, as the more general version of Corollary~\ref{1.5}.

\section{Proof of Theorem~\ref{1.2}}\label{Sec2}

In this section, we establish Theorem~\ref{1.2} in the more general context of narrow genus class characters of real quadratic fields. The argument is a direct generalization of that of Section~3 of \cite{Mok1}. Throughout this section, $N_-$ is a squarefree product of an odd number of distinct primes.

\bigskip
Thus let $\chi_1,\chi_2$ be a pair of quadratic (or trivial) primitive Dirichlet characters such that $\chi_1(-1) =\chi_2(-1)$, not both trivial, corresponding to quadratic (or trivial) field extensions of $\mathbb{Q}$ with fundamental discriminants $D_1,D_2$, i.e. $\chi_1,\chi_2$ are the Kronecker symbols for the discriminants $D_1,D_2$, whose conductors are equal to $|D_1|$ and $|D_2|$ respectively; the common sign $\chi_1(-1)=\chi_2(-1)$ will be denoted as $w$. We assume that $D_1$ and $D_2$ are relatively prime and that $D_1 \cdot D_2$ is relatively prime to the conductor of $E/\mathbb{Q}$, namely $N=N_+N_-$. Put $\kappa := \chi_1 \cdot \chi_2$. Thus $\kappa$ corresponds to the real quadratic extension $F:=\mathbb{Q}(\sqrt{D})$ of $\mathbb{Q}$, with $D:= D_1 \cdot D_2$ (thus $D$ is equal to the discriminant $D_F$ of $F$). As before the modified Heegner hypothesis with respect to $(N_+,N_-)$ is always enforced on $D$: all primes dividing $N_+$ split in $F$ (i.e. $\chi_1(l) =\chi_2(l)$ for all primes $l$ dividing $N_+$), and all primes dividing $N_-$ are inert in $F$ (i.e. $\chi_1(q) =-\chi_2(q)$ for all primes $q$ dividing $N_-$). For the proof of Theorem~\ref{3.4} in the next section, both $\chi_1,\chi_2$ are even (and thus $w=+1$), but we work with this more general context in the arguments below.

\bigskip

Below we will deal with quadratic (or trivial) Hecke characters of $F$ (and CM extensions of $F$); thus such a character $\delta$ is a $\{\pm 1 \}$-valued character of the idele class group $\mathbf{A}_F^{\times}/F^{\times}$. For a prime $v$ of $F$ (finite or archimedean) we denote by $\delta_v$ the restriction of $\delta$ to $F_v$, the completion of $F$ at $v$. Then $\delta_v$ is unramified for almost all $v$. We have $\delta = \otimes_v^{\prime} \delta_v$ (as character of $\mathbf{A}_F^{\times}$). We fix local uniformizers $\pi_v \in F_v$ for the finite primes $v$ of $F$. 

\bigskip

In particular the pair of primitive Dirichlet characters $(\chi_1,\chi_2)$ defines a narrow genus class character $\chi_F$ of $F$ (if one of $\chi_1$ or $\chi_2$ is trivial then $\chi_F$ is trivial); identified as a $\{\pm 1\}$-valued Hecke character of $F$, we have that $\chi_{F,v}$ is unramified for all finite primes $v$ (thus the conductor $\mathfrak{c}_{\chi_F}$ of $\chi_F$ is equal to $\mathcal{O}_F$), and that:

\begin{itemize}
\item $\chi_{F,v}(-1) =w$ for archimedean $v$.

\item $\chi_{F,{\mathfrak{q}}}(\pi_{\mathfrak{q}}) = 1$ for $\mathfrak{q}$ any finite prime of $F$ that lies over a prime number $q$ that is inert in $F$.

\item $\chi_{F,{\mathfrak{l}}}(\pi_{\mathfrak{l}}) = \chi_1(l)=\chi_2(l)$ for $\mathfrak{l}$ any finite prime of $F$ that lies over a prime number $l$ that splits in $F$.
\end{itemize}
In particular by the modified Heegner hypothesis, we have that $\chi_{F,{\mathfrak{l}}}(\pi_{\mathfrak{l}})=\chi_1(l) =\chi_2(l)$ for primes $l$ dividing $N_+$ and the two primes $\mathfrak{l}$ of $F$ lying above $l$, and $\chi_{F,{\mathfrak{q}}}(\pi_{\mathfrak{q}})=1$ for primes $q$ dividing $N_-$ and the prime $\mathfrak{q} = q \mathcal{O}_F$ of $F$ lying above $q$.

\bigskip

Denote by $f_E$ the weight $2$ cuspidal newform of level $N$ corresponding to $E/\mathbb{Q}$, and by $\mathbf{f}_E$ the base change of $f_E$ from $\GL_{2/\mathbb{Q}}$ to $\GL_{2/F}$. Thus $\mathbf{f}_E$ is a cuspidal Hilbert newform of parallel weight $2$ corresponding to $E/F$. In particular the $L$-function of $E/\mathbb{Q}$, resp. that of $E/F$, is equal to the $L$-function of $f_E$, resp. that of $\mathbf{f}_E$. The conductor of $E/F$ (and hence the level of $\mathbf{f}_E$) is $N \mathcal{O}_F$, as $D_F$ and $N$ are relatively prime. Also the Hecke eigenvalues of $f_E$ and $\mathbf{f}_E$ are all in $\mathbb{Z}$ as $E/\mathbb{Q}$ (and hence $E/F$) is an elliptic curve.

\bigskip
The $L$-function $L(s,E/F)$ has the following factorization:
\begin{eqnarray*}
L(s,E/F) = L(s,E/\mathbb{Q}) \cdot L(s,E/\mathbb{Q},\kappa) 
\end{eqnarray*}
and more generally that $L$-function $L(s,E/F,\chi_F)$ has the following factorization:
\begin{eqnarray*}
L(s,E/F,\chi_F) = L(s,E/\mathbb{Q},\chi_1) \cdot L(s,E/\mathbb{Q},\chi_2) 
\end{eqnarray*}

The modified Heegner hypothesis implies that the sign $\epsilon(E/\mathbb{Q},\kappa)$ of the functional equation for the $L$-function $L(s,E/\mathbb{Q},\kappa)$ is opposite to $\epsilon(E/\mathbb{Q})$, the sign of the functional equation for the $L$-function $L(s,E/\mathbb{Q})$, and consequently the sign $\epsilon(E/F)$ of the functional equation for the $L$-function $L(s,E/F)$ is equal to $-1$; more generally the sign $\epsilon(E/\mathbb{Q},\chi_1)$ of the functional equation for the $L$-function $L(s,E/\mathbb{Q},\chi_1)$ is opposite to $\epsilon(E/\mathbb{Q},\chi_2)$, the sign of the functional equation for the $L$-function $L(s,E/\mathbb{Q},\chi_2)$, and consequently the sign $\epsilon(E/F,\chi)$ of the functional equation for the $L$-function $L(s,E/F,\chi)$ is also equal to $-1$. In particular we have $L(1,E/F) = L(1,E/F,\chi)=0$.

\bigskip

Now we recall some further notations.

\bigskip
Denote by $\mathrm{height}_{\mathbb{Q}}(-)$ the N\'{e}ron-Tate height of $E/\mathbb{Q}$, which is positive definite quadratic function on $E(\overline{\mathbb{Q}})$ modulo the torsion points of $E(\overline{\mathbb{Q}})$, and is extended naturally to $E(\overline{\mathbb{Q}}) \otimes_{\mathbb{Z}} \mathbb{Q}$. Then for any number field $L$, $\mathrm{height}_L(-)$, the N\'{e}ron-Tate height of $E/L$, is given by $\mathrm{height}_L(-) = [L:\mathbb{Q}] \cdot \mathrm{height}_{\mathbb{Q}}(-)$ (in particular $\mathrm{height}_F(-) = 2 \cdot \mathrm{height}_{\mathbb{Q}}(-)$).

\bigskip

Denote by $\Omega^+_{E/\mathbb{Q}},\Omega^-_{E/\mathbb{Q}}$ the real, and respectively, imaginary N\'{e}ron periods of $E/\mathbb{Q}$ defined with respect to the N\'{e}ron differential $\omega_{\min}$ associated to a global minimal Weierstrass equation $E_{\min}$ for $E/\mathbb{Q}$; we recall the definition (as in \cite{Pal}):
\begin{eqnarray*}
\Omega^+_{E/\mathbb{Q}} &=& \int_{E_{\min}(\mathbb{R})} |\omega_{\min}| \\
\Omega^-_{E/\mathbb{Q}} &=& \int_{\gamma^{-}} \omega_{\min}
\end{eqnarray*}
where $\gamma^{-}$ is a generator of $H_1(E_{\min},\mathbb{Z})^{-}$, the subgroup of elements in $H_1(E_{\min},\mathbb{Z})$ which are negated by complex conjugation.

\bigskip

For primes $q|N_-$ denote by $c_q(E/F)$ the local Tamagawa number of $E/F$ at the prime $q \mathcal{O}_F$. As $E/\mathbb{Q}$ has multiplicative reduction at $q$ and that $q$ is inert in $F$, we have that $E/F$ has split multiplicative reduction at $q \mathcal{O}_F$, so by Tate's non-archimedean uniformization theory we have that $c_q(E/F)$ is equal to the $q\mathcal{O}_F$-valuation of the minimal discriminant of $E/F$, which is also equal to the $q$-valuation of the minimal discriminant of $E/\mathbb{Q}$ (as $q$ is inert and so in particular is unramified in $F$).

\bigskip

Finally, the narrow genus class character $\chi_F$ corresponds under class field theory to the narrow genus class field $H_{\chi_F}$, a quadratic (or trivial) extension of $F$ that is unramified at all the finite primes of $F$. Explicitly we have $H_{\chi_F}=\mathbb{Q}(\sqrt{D_1},\sqrt{D_2})$. We also regard $\chi_F$ as a character of $\Gal(H_{\chi_F}/F)$. Denote by $(E(H_{\chi_F}) \otimes_{\mathbb{Z}} \mathbb{Q})^{\chi_F}$ the $\chi_F$-eigenspace of $E(H_{\chi_F}) \otimes_{\mathbb{Z}} \mathbb{Q}$.

\bigskip

We now show the following:

\begin{theorem}\label{2.1}
With notations as above, we have:
\begin{eqnarray}\label{e2.1}
& & L^{\prime}(1,E/F,\chi_F) \\ &=& \frac{1}{\sqrt{D_F}}  \prod_{q|N_-} c_q(E/F) \cdot \mathrm{height}_F(\mathbf{P}) \cdot w \cdot (\Omega^{w}_{E/\mathbb{Q}})^2 \times \mbox{square of a nonzero rational number} \nonumber
\end{eqnarray}
for some $\mathbf{P} \in (E(H_{\chi_F}) \otimes_{\mathbb{Z}} \mathbb{Q})^{\chi_F} $. In addition if $L^{\prime}(1,E/F,\chi) \neq 0$, then $\mathrm{dim}_{\mathbb{Q}}(E(H_{\chi_F}) \otimes_{\mathbb{Z}} \mathbb{Q})^{\chi_F} =1$, and $\mathbf{P}$ spans $(E(H_{\chi_F}) \otimes_{\mathbb{Z}} \mathbb{Q})^{\chi_F} $. 
\end{theorem}
\begin{proof}
The theorem is trivial if $L^{\prime}(1,E/F,\chi_F) = 0$. Hence we assume that $L^{\prime}(1,E/F,\chi_F) \neq 0$ in the rest of the proof.

\bigskip

We first choose an auxiliary quadratic Hecke character $\delta_1$ of $F$ satisfying the following conditions:
\begin{itemize}
\item (a) $\delta_1$ is unramified at all primes of $F$ dividing $N$.
\item (b) $\delta_{1,v}(-1) = -w$ for the two archimedean primes $v$ of $F$.
\item (c) $\delta_{1,\mathfrak{q}} (\pi_{\mathfrak{q}})=-1$ for all primes $q$ dividing $N_-$ and the prime $\mathfrak{q} = q \mathcal{O}_F$ of $F$ lying above $q$.
\item (d) $\delta_{1,\mathfrak{l}} (\pi_{\mathfrak{l}})= \chi_1(l) =\chi_2(l)$ for all primes $l$ dividing $N_+$ and the two primes $\mathfrak{l}$ of $F$ lying above $l$. 
\item (e) $L(1,E/F,\delta_1) \neq 0$.
\end{itemize}
The existence of $\delta_1$ satisfying conditions (a)-(e) is guaranteed by Friedberg-Hoffstein \cite{FH}, applied to the Hilbert newform $\mathbf{f}_E$. Indeed for any quadratic Hecke character $\delta_1$ satisfying conditions (a)-(d), we have that the sign $\epsilon(E/F,\delta_1)$ of the functional equation for the $L$-function $L(s,E/F,\delta_1)=L(s,\mathbf{f}_E,\delta_1)$ is equal to
\[
\left(\prod_{v | \infty} \delta_{1,v}(-1) \right) \cdot \left(\prod_{\mathfrak{l} | N_+} \delta_{1,\mathfrak{l}}(\pi_{\mathfrak{l}}) \right)  \cdot  \left(\prod_{\mathfrak{q} | N_-} \delta_{1,\mathfrak{q}}(\pi_{\mathfrak{q}}) \right)  \cdot  \epsilon(E/F)
\]
thus it is equal to $+1$; indeed recall that $\epsilon(E/F)=-1$, and that the primes of $F$ dividing $N_-$ are all of the form $\mathfrak{q} = q \mathcal{O}_F$, with primes $q|N_-$ (and $N_-$ is a squarefree product of an odd number of such $q$'s), so in particular the number of such $\mathfrak{q}$'s is odd (and $N_-\mathcal{O}_F$ is the squarefree product of these $\mathfrak{q}$'s). Hence by \cite{FH}, such a $\delta_1$ satisfying (a)-(e) exists. Fix one such $\delta_1$, and denote by $\mathfrak{c}_{\delta_1}$ the conductor of $\delta_1$.

\bigskip
The quadratic Hecke character $\chi_F \cdot \delta_1$ then corresponds to a CM extension $K$ of $F$, with the property that all primes of $F$ that divide $N_-\mathcal{O}_F$ are inert in $K$, and that all primes of $F$ that divide $N_+\mathcal{O}_F$ split in $K$. In addition the pair $(\chi_F,\delta_1)$ defines a genus class character $\delta_K$ of $K$, which corresponds to the genus class field $H_{\delta_K}$ of $K$, which is a quadratic (or trivial) extension of $K$ that is unramified at all the primes of $K$; explicitly, denoting by $K_1$ the quadratic extension of $F$ that corresponds to $\delta_1$, we have that $H_{\delta_K}$ is the composite of $H_{\chi}$ with $K_1$. The character $\delta_K$ is also regarded as a character of $\Gal(H_{\delta_K}/K)$. We then have the following factorization of $L$-functions:
\begin{eqnarray*}
L(s,E/K,\delta_K) = L(s,E/F,\chi_F) \cdot L(s,E/F,\delta_1)
\end{eqnarray*}
In particular,
\begin{eqnarray}\label{e2.2}
L^{\prime}(1,E/K,\delta_K) = L^{\prime}(1,E/F,\chi_F) \cdot L(1,E/F,\delta_1) \neq 0.
\end{eqnarray}

\bigskip
We now apply the generalized Gross-Zagier formula of Zhang \cite{Zhang1,Zhang2}, in the explicit form given in Theorem 1.5 of \cite{CST}, to the pair $E/F$ and $\delta_K$ and obtain:
\begin{eqnarray}\label{e2.3}
L^{\prime}(1,E/K,\delta_K) = \frac{(8 \pi^2)^2}{(\mathcal{N}_{F/\mathbb{Q}} D_{K/F})^{1/2}} \cdot  \frac{\langle \mathbf{f}_E,\mathbf{f}_E   \rangle}{\deg_{E/F}} \cdot \mathrm{height}_K (\mathbf{P}_{\delta_K}).  
\end{eqnarray}

The explanation of these terms is as follows. Here $D_{K/F}$ is the relative discriminant of $K/F$ and $\mathcal{N}_{F/\mathbb{Q}} D_{K/F}$ is its norm, while $\langle \mathbf{f}_E,\mathbf{f}_E   \rangle$ is the Petersson inner product of $\mathbf{f}_E$ with itself. As for the other terms, we fix a modular parametrization over $F$ of $E/F$ by the Shimura curve $X(N_+\mathcal{O}_F,N_-\mathcal{O}_F)/F$. Here $X(N_+\mathcal{O}_F,N_-\mathcal{O}_F)/F$ is the Shimura curve over $F$ with Eichler level $N_+ \mathcal{O}_F$ associated to the quaternion algebra $\mathcal{B}$ over $F$, with ramification locus given by $N_- \mathcal{O}_F$ and exactly one of the two archimedean places of $F$. Then $\deg_{E/F}$ is the degree of this modular parametrization, and $\mathbf{P}_{\delta_K}$ is the Heegner point associated to the genus class character $\delta_K$ (with respect to this chosen modular parametrization); more precisely $\mathbf{P}_{\delta_K} \in (E(H_{\delta_K})\otimes_{\mathbb{Z}} \mathbb{Q} )^{\delta_K} $. In particular $\mathbf{P}_{\delta_K}$ is non-torsion as $L^{\prime}(1,E/K,\delta_K) \neq 0$.

\bigskip

Recall that $K_1$ is the quadratic extension of $F$ that corresponds to the character $\delta_1$; we also regard $\delta_1$ as a character of $\Gal(K_1/F)$. Now as $L(1,E/F,\delta_1) \neq 0$, so by Theorem A of \cite{Zhang1} (following the methods of Kolyvagin-Logachev \cite{KL}) we have that 
\[
(E(K_1)\otimes_{\mathbb{Z}} \mathbb{Q} )^{\delta_1}=\{0\}.
\]
and so by the same argument as in pp. 1946--1947 of \cite{Mok1}, we then deduce from this that
\[
(E(H_{\delta_K})\otimes_{\mathbb{Z}} \mathbb{Q} )^{\delta_K} = (E(H_{\chi})\otimes_{\mathbb{Z}} \mathbb{Q} )^{\chi_F}.
\]

\bigskip
Hence the point $\mathbf{P}_{\delta_K}$ can be regarded as an element of $(E(H_{\chi_F})\otimes_{\mathbb{Z}} \mathbb{Q} )^{\chi_F}$, and for the Equation \eqref{e2.1} to be proved, we take $\mathbf{P}$ be equal to this $\mathbf{P}_{\delta_K}$. In addition as $L^{\prime}(1,E/K,\delta_K) \neq 0$, so again Theorem A of \cite{Zhang1} asserts that:
\[
\dim_{\mathbb{Q}} (E(H_{\delta_K})\otimes_{\mathbb{Z}} \mathbb{Q} )^{\delta_K} =1
\]
and thus $\mathbf{P}$ spans the one dimensional space $(E(H_{\chi_F})\otimes_{\mathbb{Z}} \mathbb{Q} )^{\chi_F}$. 

\bigskip

We next choose another auxiliary quadratic Hecke character $\delta$ of $F$ satisfying the following conditions:
\begin{itemize}
\item (a) $\delta$ is unramified at all primes dividing $N \mathfrak{c}_{\delta_1}$.
\item (b) $\delta_{v}(-1) = w$ for the two archimedean primes $v$ of $F$.
\item (c) $\delta_{\mathfrak{q}} (\pi_{\mathfrak{q}})=-1$ for primes $q$ dividing $N_-$ and the prime $\mathfrak{q} =q \mathcal{O}_F$ of $F$ lying above $q$. 
\item (d) $\delta_{\mathfrak{l}} (\pi_{\mathfrak{l}})= \chi_1(l) =\chi_2(l)$ for primes $l$ dividing $N_+$ and the two primes $\mathfrak{l}$ of $F$ lying above $l$. 
\item (e) $L(1,E/F,\delta) \neq 0$.
\end{itemize}
The existence of $\delta$ satisfying conditions (a)-(e) is again guaranteed by \cite{FH}. Indeed for any quadratic Hecke character $\delta$ satisfying conditions (a)-(d), we have, by similar arguments as before, that the sign of the functional equation for $L(s,E/F,\delta)=L(s,\mathbf{f}_E,\delta)$ is equal to $+1$; hence by \cite{FH}, such a $\delta$ satisfying (a)-(e) exists. Fix one such $\delta$, and denote by $\mathfrak{c}_{\delta}$ the conductor of $\delta$. In particular $\mathfrak{c}_{\delta}$ and $\mathfrak{c}_{\delta_1}$ are relatively prime.

\bigskip
The quadratic Hecke character $\delta \cdot \delta_1$ then corresponds to a CM extension $\widetilde{K}$ of $F$, with the property that all primes of $F$ that divide $N = N_+ N_-$ are split in $\widetilde{K}$. In addition the pair $(\delta,\delta_1)$ defines in a similar way a genus class character $\delta_{\widetilde{K}}$ of $\widetilde{K}$ (thus denoting by $H_{\delta_{\widetilde{K}}}$ the genus class field of $\widetilde{K}$ corresponding $\delta_{\widetilde{K}}$, we have that $H_{\delta_{\widetilde{K}}}$ is the composite of $K_1$ with the quadratic extension of $F$ corresponding to the character $\delta$). We then have the following factorization of $L$-functions:
\begin{eqnarray*}
L(s,E/\widetilde{K},\delta_{\widetilde{K}}) = L(s,E/F,\delta) \cdot L(s,E/F,\delta_1)
\end{eqnarray*}
so in particular
\begin{eqnarray}\label{e2.4}
L(1,E/\widetilde{K},\delta_{\widetilde{K}}) = L(1,E/F,\delta) \cdot L(1,E/F,\delta_1) \neq 0.
\end{eqnarray}

\bigskip
Now as in pp. 1948 of \cite{Mok1}, we apply Zhang's central value formula \cite{Zhang2} in the explicit form given in Theorem 1.10 of \cite{CST}, to the pair $E/F$ and $\delta_{\widetilde{K}}$ and obtain:
\begin{eqnarray}\label{e2.5}
& & L(1,E/\widetilde{K},\delta_{\widetilde{K}}) \\
&=& \frac{(8 \pi^2)^2}{(\mathcal{N}_{F/\mathbb{Q}} D_{\widetilde{K}/F})^{1/2}} \cdot  \frac{\langle \mathbf{f}_E,\mathbf{f}_E   \rangle}{\langle  \Phi_E, \Phi_E \rangle } \cdot \mbox{square of a non-zero rational number}.  \nonumber
\end{eqnarray}
Here $D_{\widetilde{K}/F}$ is the relative discriminant of $\widetilde{K}/F$ and $\mathcal{N}_{F/\mathbb{Q}} D_{\widetilde{K}/F}$ is its norm; $\langle \mathbf{f}_E,\mathbf{f}_E   \rangle$ is the Petersson inner product of $\mathbf{f}_E$ with itself as before. As for the term $\langle  \Phi_E, \Phi_E \rangle$, we first define $B_F$ to be the totally definite quaternion algebra over $F$ that ramifies exactly at the two archimedean primes of $F$ (and is split at all the finite primes of $F$). We denote by $\Phi_E$ the scalar-valued automorphic eigenform with trivial central character for the group $B_F^{\times}$, with Eichler level $N \mathcal{O}_F$, that corresponds to the cuspidal Hilbert newform $\mathbf{f}_E$ under the Jacquet-Langlands correspondence. We normalize $\Phi_E$ by requiring that the values taken by the automorphic form $\Phi_E$ all lie in $\mathbb{Q}$ (which is possible because the Hecke eigenvalues of $\mathbf{f}_E$ are all in $\mathbb{Z}$). With this normalization $\Phi_E$ is then well-defined up to $\mathbb{Q}^{\times}$-multiples. Then $ \langle \Phi_E, \Phi_E \rangle$ is the Petersson inner product of $\Phi_E$ with itself (which is thus a positive rational number, whose class $\mod{(\mathbb{Q}^{\times})^2}$ is independent of the choice of the $\mathbb{Q}^{\times}$-multiple).

\bigskip

Now on combining \eqref{e2.2}-\eqref{e2.5}, we thus obtain:
\begin{eqnarray}\label{e2.6}
& & L^{\prime}(1,E/F,\chi_F) \\
&=& L^{\prime}(1,E/K,\delta_{K}) \cdot  L(1,E/F,\delta)  \cdot L(1,E/\widetilde{K},\delta_{\widetilde{K}})^{-1} \nonumber \\
&=&   \frac{\langle \Phi_E,\Phi_E \rangle}{\deg_{E/F}}  \cdot \sqrt{\frac{\mathcal{N}_{F/\mathbb{Q}}D_{\widetilde{K}/F}}{\mathcal{N}_{F/\mathbb{Q}}D_{K/F}}  } \cdot L(1,E/F,\delta) \cdot \mathrm{height}_K(\mathbf{P})  \mod{(\mathbb{Q}^{\times})^2} \nonumber 
\end{eqnarray}

\bigskip
Now we have the conductor-discriminant identities ({\it c.f.} pp. 1404 of \cite{Mok2}):
\begin{eqnarray}\label{e2.7}
& & \mathcal{N}_{F/\mathbb{Q}}D_{\widetilde{K}/F} = \mathcal{N}_{F/\mathbb{Q}} \mathfrak{c}_{\delta} \cdot \mathcal{N}_{F/\mathbb{Q}} \mathfrak{c}_{\delta_1}, \\
& & \mathcal{N}_{F/\mathbb{Q}} D_{K/F} = \mathcal{N}_{F/\mathbb{Q}}\mathfrak{c}_{\chi_F} \cdot \mathcal{N}_{F/\mathbb{Q}}\mathfrak{c}_{\delta_1}=\mathcal{N}_{F/\mathbb{Q}}\mathfrak{c}_{\delta_1}. \nonumber
\end{eqnarray}

\bigskip
At this point we apply Theorem 3.2 of \cite{Mok2} to deal with the $L$-value term $L(1,E/F,\delta)$ (with $M$ and $\mathcal{Q}$ of {\it loc. cit.} being taken to be equal to $N_+$ and $N_-$ respectively); here we note that Theorem 3.2 of \cite{Mok2} is indeed applicable, as we have $L^{\prime}(1,E/F,\chi_F) \neq 0$. Thus, according to {\it loc. cit.} we have:
\begin{eqnarray}\label{e2.8}
L(1,E/F,\delta) = \frac{\tau(\delta) \cdot (\Omega_{E/\mathbb{Q}}^w)^2}{2 D_F^{1/2}  \cdot \mathcal{N}_{F/\mathbb{Q}}\mathfrak{c}_{\delta}  }  \mod{(\mathbb{Q}^{\times})^2}
\end{eqnarray}
here $\tau(\delta)$ is the Gauss sum of the quadratic Hecke character $\delta$.

\bigskip

Thus on combining \eqref{e2.6} - \eqref{e2.8}, we obtain:
\begin{eqnarray}\label{e2.9}
& & L^{\prime}(1,E/F,\chi_F) \\& =&  \frac{1}{2 D_F^{1/2}} \cdot\frac{\langle \Phi_E,\Phi_E \rangle}{\deg_{E/F}} \cdot \frac{\tau(\delta)}{(\mathcal{N}_{F/\mathbb{Q}} \mathfrak{c}_{\delta} )^{1/2}}  \cdot \mathrm{height}_K(\mathbf{P}) \cdot (\Omega_{E/\mathbb{Q}}^w)^2   \mod{(\mathbb{Q}^{\times})^2}. \nonumber
\end{eqnarray}

\noindent By using the Gauss sum identity ({\it c.f.} pp. 1404 of \cite{Mok2}):
\[
 \frac{\tau(\delta)}{(\mathcal{N}_{F/\mathbb{Q}} \mathfrak{c}_{\delta} )^{1/2}} = w,
\]
we then obtain:
\begin{eqnarray}\label{e2.10}
   & & L^{\prime}(1,E/F,\chi_F) \\& =&  \frac{1}{2 D_F^{1/2}} \cdot\frac{\langle \Phi_E,\Phi_E \rangle}{\deg_{E/F}} \cdot  \mathrm{height}_K(\mathbf{P}) \cdot w \cdot (\Omega_{E/\mathbb{Q}}^w)^2   \mod{(\mathbb{Q}^{\times})^2}. \nonumber 
\end{eqnarray}

Finally, we deal with the term $\langle \Phi_E,\Phi_E \rangle / \deg_{E/F}$. For this we apply Lemma 3.4 of \cite{Mok2}, specifically Equation \eqref{e3.19} of {\it loc. cit.} (again with the $M$ and $\mathcal{Q}$ there being taken to be equal to $N_+$ and $N_-$ respectively) which tells us that:
\begin{eqnarray}\label{e2.11}
\deg_{E/F} = \left( \prod_{q|N_-} c_q(E/F) \right) \cdot \langle \Phi_E,\Phi_E \rangle \mod{(\mathbb{Q}^{\times})^2}
\end{eqnarray}

\bigskip

Finally we have:
\begin{eqnarray}\label{e2.12}
\mathrm{height}_K(\mathbf{P}) = 2 \cdot \mathrm{height}_F(\mathbf{P})
\end{eqnarray}
and on combining \eqref{e2.10} - \eqref{e2.12}, we see that Theorem~\ref{2.1} is proved.

\end{proof}

\bigskip

Now let $E^{D_1}/\mathbb{Q}$ and $E^{D_2}/\mathbb{Q}$ be the quadratic twists of $E/\mathbb{Q}$ by the fundamental discriminants $D_1$ and $D_2$ respectively. The $L$-function $L(s,E/\mathbb{Q},\chi_i)$ coincides with the $L$-function $L(s,E^{D_i}/\mathbb{Q})$ for $i=1,2$. In particular:
\begin{eqnarray}
L(s,E/F,\chi_F)  = L(s,E^{D_1}/\mathbb{Q}) \cdot L(s,E^{D_2}/\mathbb{Q})
\end{eqnarray}

\bigskip
By Kolyvagin again \cite{K}, we have $E^{D_2}(\mathbb{Q}) \otimes_{\mathbb{Z}} \mathbb{Q} = \{0\}$ (resp. $E^{D_1}(\mathbb{Q}) \otimes_{\mathbb{Z}} \mathbb{Q} = \{0\}$) if $L(1,E^{D_2}/\mathbb{Q}) \neq 0$ (resp. $L(1,E^{D_1}/\mathbb{Q}) \neq 0$), in which case we have $(E(H_{\chi_F}) \otimes_{\mathbb{Z}} \mathbb{Q})^{\chi_F}  =  E^{D_1}(\mathbb{Q}) \otimes_{\mathbb{Z}} \mathbb{Q}$ (resp. $(E(H_{\chi_F}) \otimes_{\mathbb{Z}} \mathbb{Q})^{\chi_F}  =  E^{D_2}(\mathbb{Q}) \otimes_{\mathbb{Z}} \mathbb{Q}$), again by the same argument as in pp. 1946-1947 of \cite{Mok1}, or Theorem 4.7 of \cite{BD}. Hence we can take $\mathbf{P}$ in Theorem~\ref{2.1} to be an element of $E^{D_1}(\mathbb{Q}) \otimes_{\mathbb{Z}} \mathbb{Q}$ (resp. an element of $E^{D_2}(\mathbb{Q}) \otimes_{\mathbb{Z}} \mathbb{Q}$) if $L(1,E^{D_2}/\mathbb{Q}) \neq 0$ (resp. $L(1,E^{D_1}/\mathbb{Q}) \neq 0$). Finally if in addition we have $L^{\prime}(1,E^{D_1}/\mathbb{Q}) \neq 0$ (resp. $L^{\prime}(1,E^{D_2}/\mathbb{Q}) \neq 0$), then we have:
\[
L^{\prime}(1,E/F,\chi_F) = L^{\prime}(1,E^{D_1}/\mathbb{Q}) \cdot L(1,E^{D_2}/\mathbb{Q}) \neq 0
\]
(resp. $L^{\prime}(1,E/F,\chi_F) = L(1,E^{D_1}/\mathbb{Q}) \cdot L^{\prime}(1,E^{D_2}/\mathbb{Q}) \neq 0$), and so $(E(H_{\chi_F}) \otimes_{\mathbb{Z}} \mathbb{Q})^{\chi_F}$ is one dimensional, and consequently $E^{D_1}(\mathbb{Q}) \otimes_{\mathbb{Z}} \mathbb{Q}$ (resp. $E^{D_2}(\mathbb{Q}) \otimes_{\mathbb{Z}} \mathbb{Q}$) is one dimensional.

\bigskip

Finally note that Theorem~\ref{1.2} is then a special case of Theorem~\ref{2.1} and the above discussions, by taking $\chi_1$ to be trivial, and $\chi_2$ to be the Kronecker symbol for the positive fundamental discriminant $D$, in particular $w=+1$ (in this case we have that $\chi_F$ is trivial and $H_{\chi_F}=F$).

\section{Proof of Theorem~\ref{1.1} modulo Theorem~\ref{1.3}}\label{Sec3}

In this section, we prove Theorem~\ref{1.1}; in fact, we establish a more general version of it, namely Theorem~\ref{3.4}. The proof is modulo Theorem~\ref{1.3}, which we will establish in Section~\ref{Sec4} (which is independent of the previous sections). As we have already said in the introduction, Theorem~\ref{1.3} only requires $N_-$ to be squarefree and relatively prime to $N_+$ (in particular $N_-$ can be equal to one); in fact the arguments and results in this section will also be applied in Section~\ref{Sec5} to treat the case where $N_-$ is a squarefree product of an even number of distinct prime factors. Thus, for the moment, we return to the situation where $N_-$ is only required to be squarefree (and relatively prime to $N_+$). 

\bigskip

Thus the conductor $N$ of $E/\mathbb{Q}$ is factored as $N=N_+ N_-$, with $N_+,N_-$ relatively prime and $N_-$ squarefree. We consider quadratic (or trivial) even primitive Dirichlet characters $\chi_1,\chi_2$, not both trivial, whose conductors are noted as $D_1,D_2$; thus $D_1$ is the (positive) fundamental discriminant for $\mathbb{Q}(\sqrt{D_1})$, and similarly for $D_2$. We assume that $D_1,D_2$ are relatively prime. The even quadratic primitive Dirichlet character $\kappa:=\chi_1 \cdot \chi_2$ then has conductor $D:=D_1 \cdot D_2$. Again put $F=\mathbb{Q}(\sqrt{D})$ (and so $D_F=D$), and the pair $(\chi_1,\chi_2)$ gives a genus character $\chi_F$ of $F$ as in the previous section (with $w=+1$). As before we always assume that $D$ satisfies the modified Heegner hypothesis with respect to $(N_+,N_-)$, i.e. $\chi_1(l)=\chi_2(l)$ for all primes $l|N_+$, and $\chi_1(q)=-\chi_2(q)$ for all primes $q|N_-$. In addition, we assume the following condition holds for the pair $(\chi_1,\chi_2)$:

\bigskip
\noindent (*) If a prime $l|N$ is such that either one of $\chi_1(l)$ or $\chi_2(l)$ is equal to $-1$, then $l$ exactly divides $N$, i.e. that $E/\mathbb{Q}$ has multiplicative reduction at $l$.
\bigskip

\noindent Of course, this condition is automatic if $l|N_-$; also, if one of $\chi_1$ or $\chi_2$ is trivial, then this condition follows already from the modified Heegner hypothesis for $D$ with respect to $(N_+,N_-)$. 

\bigskip

Below, for primes $l|N$, we denote by $n_l$ the exact power of $l$ dividing $N$ (i.e., the $l$-valuation of $N$). We now define the following quantities: 

\bigskip

\begin{eqnarray*}
N_{1,+}^I &:=&  \prod_{l|N_+,\chi_1(l)=1} l^{n_l} \\
N_{1,-}^I &:=&  \prod_{l|N_+,\chi_1(l)=-1} l \\
N_{1,+}^{II} &:=&  \prod_{l|N_-,\chi_1(l)=1} l \\
N_{1,-}^{II} &:=&  \prod_{l|N_-,\chi_1(l)=-1} l
\end{eqnarray*}

The integers $N_{1,+}^I,N_{1,-}^I,N_{1,+}^{II},N_{1,-}^{II}$ are then pairwise relatively prime, and by definition $N_{1,-}^I,N_{1,+}^{II},N_{1,-}^{II}$ are squarefree. Define similarly:

\begin{eqnarray*}
N_{2,+}^I &:=&  \prod_{l|N_+,\chi_2(l)=1} l^{n_l} \\
N_{2,-}^I &:=&  \prod_{l|N_+,\chi_2(l)=-1} l \\
N_{2,+}^{II} &:=&  \prod_{l|N_-,\chi_2(l)=1} l \\
N_{2,-}^{II} &:=&  \prod_{l|N_-,\chi_2(l)=-1} l
\end{eqnarray*}

The integers $N_{2,+}^I,N_{2,-}^I,N_{2,+}^{II},N_{2,-}^{II}$ are again pairwise relatively prime, and by definition $N_{2,-}^I,N_{2,+}^{II},N_{2,-}^{II}$ are squarefree.
\bigskip

\noindent By the modified Heegner hypothesis for $D$ with respect to $(N_+,N_-)$, we then have:
\begin{eqnarray*}
N_{1,+}^I = N_{2,+}^I, \,\ N_{1,-}^I = N_{2,-}^I,  \,\ N_{1,+}^{II} = N_{2,-}^{II}, \,\ N_{1,-}^{II} = N_{2,+}^{II} 
\end{eqnarray*}
In particular $N_{1,-}^{II}$ and $N_{2,-}^{II}$ are relatively prime. 

\bigskip
Define $N_{1,+} :=N_{1,+}^I \cdot N_{1,+}^{II}, \,\ N_{2,+} :=N_{2,+}^I \cdot N_{2,+}^{II}, \,\ N_{1,-} :=N_{1,-}^I \cdot N_{1,-}^{II}, \,\ N_{2,-} :=N_{2,-}^I \cdot N_{2,-}^{II}$. Then $N_{1,-}$ and $N_{2,-}$ are squarefree, the pair $(N_{1,+},N_{1,-})$ is relatively prime, and similarly the pair $(N_{2,+},N_{2,-})$ is relatively prime. We have $N_- = N_{1,+}^{II} \cdot N_{1,-}^{II}  = N_{2,+}^{II} \cdot N_{2,-}^{II}  = N_{1,-}^{II} \cdot N_{2,-}^{II} $. In addition, condition (*) also gives:
\[
N_+ = N_{1,+}^{I} \cdot N_{1,-}^{I}  = N_{2,+}^{I} \cdot N_{2,-}^{I}
\]
and so we have
\[
N=N_+ N_- =  N_{1,+} \cdot N_{1,-} = N_{2,+} \cdot N_{2,-}
\]

Finally note that $D_1$ satisfies the modified Heegner hypothesis with respct to $(N_{1,+},N_{1,-})$: $\chi_1(l)=1$ for all primes $l|N_{1,+}$, and $\chi_1(l)=-1$ for all primes $l|N_{1,-}$. Similarly that $D_2$ satisfies the modified Heegner hypothesis with respect to $(N_{2,+},N_{2,-})$: $\chi_2(l)=1$ for all primes $l|N_{2,+}$, and $\chi_2(l)=-1$ for all primes $l|N_{2,-}$.

\bigskip

Below we consider the quadratic twists $E^{D_1}/\mathbb{Q}$ and $E^{D_2}/\mathbb{Q}$ of $E/\mathbb{Q}$ by $D_1$ and $D_2$ respectively. Now, as in the introduction, by Proposition 2.5 and Theorem 3.2 of \cite{Pal} we have (for $i=1,2$):
\begin{eqnarray}\label{e3.1}
\Omega_{E^{D_i}/\mathbb{Q}}^+ = \frac{u_{D_i}}{\sqrt{D_i}}  \Omega_{E/\mathbb{Q}}^+
\end{eqnarray}
with $u_{D_i} \in \{1,2,4,8\}$. We remark that $D_i$ (for $i=1,2$) is a positive fundamental discriminant here (that is relatively prime to $N$), and so when compared to the notations of Proposition 2.5 and Theorem 3.2 of \cite{Pal}, we have that $u_{D_i} = \widetilde{u}=1$ if $D_i$ is odd, i.e. $D_i=d$ with $d \equiv 1 \bmod{4}$, positive squarefree, and $u_{D_i} = 2 \widetilde{u}$ if $D_i$ is even, i.e. $D_i=4 d$ with $d \equiv 2,3 \bmod{4}$, positive squarefree.

\bigskip
Also as in the introduction for primes $q|N_-$, and similarly for $q|N_{1,-}$, or $q|N_{2,-}$, we have that $E/\mathbb{Q}$ has multiplicative reduction at $q$ (split or non-split), and we define $\widetilde{c}_q(E/\mathbb{Q})$ to be equal to one (resp. two) if the $q$-valuation of the minimal discriminant of $E/\mathbb{Q}$ is odd (resp. even).

\bigskip
Recall also that for $l|D_i$ we have $c_l(E^{D_i}/\mathbb{Q}) \in \{1,2,4\}$ for $i=1,2$. We now first use Theorem~\ref{1.3} to show:
\begin{theorem}\label{3.1}
The quantity:
\begin{eqnarray}\label{e3.2}
\frac{u_{D_1}  u_{D_2}}{2^{\omega(N_-)}}  \prod_{l |D_1} c_l(E^{D_1}/\mathbb{Q})  \prod_{l |D_2} c_l(E^{D_2}/\mathbb{Q}) \prod_{q |N_-} \widetilde{c}_q(E/\mathbb{Q})
\end{eqnarray}
is an even power of two.
\end{theorem}
\begin{proof}
(Modulo the proof of Theorem~\ref{1.3}) We apply Theorem~\ref{1.3} for the discriminant $D_1$ with respect to the pair $(N_{1,+},N_{1,-})$, and similarly the discriminant $D_2$ with respect to the pair $(N_{2,+},N_{2,-})$, thus obtain that the quantities:
\begin{eqnarray}\label{e3.3}
\frac{u_{D_1}}{2^{\omega(N_{1,-})}}  \prod_{l |D_1}  c_l(E^{D_1}/\mathbb{Q}) \prod_{q |N_{1,-}} \widetilde{c}_q(E/\mathbb{Q})
\end{eqnarray}
and 
\begin{eqnarray}\label{e3.4}
\frac{u_{D_2}}{2^{\omega(N_{2,-})}}  \prod_{l |D_2}  c_l(E^{D_2}/\mathbb{Q}) \prod_{q |N_{2,-}} \widetilde{c}_q(E/\mathbb{Q})
\end{eqnarray}
are even powers of two.

\bigskip

Now as $N_{1,-} = N_{1,-}^I \cdot N_{1,-}^{II}$ and $N_{2,-} = N_{2,-}^I \cdot N_{2,-}^{II} =  N_{1,-}^I \cdot N_{2,-}^{II}$, and also that $N_- =  N_{1,-}^{II} \cdot  N_{2,-}^{II}$, we have:

\bigskip

\begin{eqnarray}\label{e3.5}
& & \omega(N_{1,-}) + \omega(N_{2,-}) \\
&=& \omega(N_{1,-}^{II})  + \omega(N_{2,-}^{II}) + 2 \omega(N_{1,-}^I) \nonumber \\
&=&  \omega(N_-) \bmod{2} \nonumber
\end{eqnarray}
and
\bigskip
\begin{eqnarray}\label{e3.6}
& & \prod_{q |N_{1,-}} \widetilde{c}_q(E/\mathbb{Q}) \cdot \prod_{q |N_{2,-}} \widetilde{c}_q(E/\mathbb{Q})  \\
&=& \prod_{q |N^{II}_{1,-}} \widetilde{c}_q(E/\mathbb{Q}) \cdot \prod_{q |N^{II}_{2,-}} \widetilde{c}_q(E/\mathbb{Q})  \cdot  \left( \prod_{q |N^{I}_{1,-}} \widetilde{c}_q(E/\mathbb{Q}) \right)^2  \nonumber \\
&=&\prod_{q |N_{-}} \widetilde{c}_q(E/\mathbb{Q}) \times \mbox{ even power of } 2 \nonumber
\end{eqnarray}

\bigskip

Thus \eqref{e3.2} is, up to an even power of two, the product of \eqref{e3.3} and \eqref{e3.4}. Hence \eqref{e3.2} is also an even power of two.
\end{proof}

\bigskip

As before for primes $q|N_-$, the quantity $c_q(E/F)$ is the local Tamagawa number of $E/F$ at the prime $\mathfrak{q} = q\mathcal{O}_F$ of $F=\mathbb{Q}(\sqrt{D})$ above $q$.

\bigskip
\begin{lemma}\label{3.2}
For primes $q|N_-$ we have:
\begin{eqnarray}\label{e3.7}
\widetilde{c}_q(E/\mathbb{Q})  \cdot c_q(E/F)  = c_q(E^{D_1}/\mathbb{Q}) \cdot c_q(E^{D_2}/\mathbb{Q})
\end{eqnarray}
\end{lemma}
\begin{proof}
Fix a prime $q|N_-$. We have $\chi_1(q) = -\chi_2(q)$. By symmetry, we may assume without loss of generality in this proof that $\chi_1(q) =1,\chi_2(q)=-1$, i.e. $q$ splits in $F_1:=\mathbb{Q}(\sqrt{D_1})$ and is inert in $F_2:=\mathbb{Q}(\sqrt{D_2})$.

\bigskip
Now $E/\mathbb{Q}$ has multiplicative reduction at $q$, so there are two cases to consider: split multiplicative and non-split multiplicative. Put $\mathfrak{q}=q\mathcal{O}_F$ (the unique prime of $F=\mathbb{Q}(\sqrt{D})$ above $q$); similarly put $\mathfrak{q}_2=q\mathcal{O}_{F_2}$ (the unique prime of $F_2$ above $q$).

\bigskip
We first consider the situation where $E/\mathbb{Q}$ has split multiplicative at $q$. By Tate's non-archimedean uniformization theory, the local Tamagawa number $c_q(E/\mathbb{Q})$ is equal to the $q$-valuation of $\Delta_{\mathrm{min}}(E/\mathbb{Q})$, the minimal discriminant of $E/\mathbb{Q}$. But $q$ is inert, and so in particular is unramified in $F$, so it follows that the $q$-valuation of $\Delta_{\mathrm{min}}(E/\mathbb{Q})$ is equal to the $\mathfrak{q}$-valuation of $\Delta_{\mathrm{min}}(E/F)$, the minimal discriminant of $E/F$. Thus we have $c_q(E/\mathbb{Q}) = c_q(E/F)$. Also $\chi_1(q)=1$ and so $E^{D_1}/\mathbb{Q}_q$ and $E/\mathbb{Q}_q$ are isomorphic over $\mathbb{Q}_q$ (in particular $E^{D_1}/\mathbb{Q}$ has split multiplicative reduction at $q$). Thus $c_q(E^{D_1}/\mathbb{Q}) = c_q(E/\mathbb{Q})$. Consequently $c_q(E^{D_1}/\mathbb{Q})= c_q(E/F)$.

\bigskip
On the other hand, as $E/\mathbb{Q}$ has split multiplicative reduction at $q$ and $\chi_2(q)=-1$, we have that $E^{D_2}/\mathbb{Q}$ has non-split multiplicative reduction at $q$. By Tate's non-archimedean uniformization theory again, the local Tamagawa number $c_q(E^{D_2}/\mathbb{Q})$ is equal to one (resp. two), if the $q$-valuation of $\Delta_{\mathrm{min}}(E^{D_2}/\mathbb{Q})$ is odd (resp. even). But by same token, the $q$-valuation of $\Delta_{\mathrm{min}}(E^{D_2}/\mathbb{Q})$ is equal to the $\mathfrak{q}_2$-valuation of $\Delta_{\mathrm{min}}(E^{D_2}/F_2)$ (the minimal discriminant of $E^{D_2}/F_2$); but this is the same as the $\mathfrak{q}_2$-valuation of $\Delta_{\mathrm{min}}(E/F_2)$, as $E/F_2$ and $E^{D_2}/F_2$ are isomorphic over $F_2$, thus again is equal to the $q$-valuation of $\Delta_{\mathrm{min}}(E/\mathbb{Q})$. Hence the $q$-valuation of $\Delta_{\mathrm{min}}(E^{D_2}/\mathbb{Q})$ is the same as the $q$-valuation of $\Delta_{\mathrm{min}}(E/\mathbb{Q})$. Thus we conclude that $c_q(E^{D_2}/\mathbb{Q}) = \widetilde{c}_q(E/\mathbb{Q})$, and so \eqref{e3.7} is verified in the case where $E/\mathbb{Q}$ has split multiplicative reduction at $q$.

\bigskip

Next, we consider the situation that $E/\mathbb{Q}$ is of non-split multiplicative reduction at $q$. This time we then have that $E^{D_2}/\mathbb{Q}$ is of split multiplicative reduction at $q$, and $E^{D_1}/\mathbb{Q}$ is of non-split multiplicative reduction at $q$. By similar arguments, we have:

\[
c_q(E^{D_1}/\mathbb{Q}) = \widetilde{c}_q(E/\mathbb{Q})
\]
\[
c_q(E^{D_2}/\mathbb{Q}) =c_q(E/F)
\]

\bigskip
\noindent and so \eqref{e3.7} is also verified in the case where $E/\mathbb{Q}$ is of non-split multiplicative reduction case at $q$. This finishes the proof of Lemma~\ref{3.2}.

\end{proof}

\begin{lemma}\label{3.3}
We have
\begin{eqnarray}\label{e3.8}
\prod_{l |N}  c_l(E^{D_1} /\mathbb{Q})  \prod_{l |N}  c_l(E^{D_2} /\mathbb{Q}) =  \prod_{q |N_-} \widetilde{c}_q(E/\mathbb{Q})  \cdot c_q(E/F) \mod{(\mathbb{Q}^{\times})^2}
\end{eqnarray}
\end{lemma}
\begin{proof}
We have   
\begin{eqnarray}\label{e3.9}
& & \prod_{l |N}  c_l(E^{D_1} /\mathbb{Q})   \\
&=&  \prod_{l |N_{1,+}}  c_l(E^{D_1} /\mathbb{Q})  \prod_{l |N_{1,-}}  c_l(E^{D_1} /\mathbb{Q})  \nonumber \\
& =&  \prod_{l |N^I_{1,+}}  c_l(E^{D_1} /\mathbb{Q})  \prod_{l |N^I_{1,-}}  c_l(E^{D_1} /\mathbb{Q})  \prod_{l |N^{II}_{1,+}}  c_l(E^{D_1} /\mathbb{Q})  \prod_{l |N^{II}_{1,-}}  c_l(E^{D_1} /\mathbb{Q})\nonumber \\
& =&  \prod_{l |N^I_{1,+}}  c_l(E^{D_1} /\mathbb{Q})  \prod_{l |N^I_{1,-}}  c_l(E^{D_1} /\mathbb{Q})  \prod_{q |N_- } c_q(E^{D_1} /\mathbb{Q})  \nonumber 
\end{eqnarray}
and similarly, we have
\begin{eqnarray}\label{e3.10}
& & \prod_{l |N}  c_l(E^{D_2} /\mathbb{Q})   \\
& =&  \prod_{l |N^I_{2,+}}  c_l(E^{D_2} /\mathbb{Q})  \prod_{l |N^I_{2,-}}  c_l(E^{D_2} /\mathbb{Q})  \prod_{q |N_- } c_q(E^{D_2} /\mathbb{Q}) \nonumber  \\
&=&   \prod_{l |N^I_{1,+}}  c_l(E^{D_2} /\mathbb{Q})  \prod_{l |N^I_{1,-}}  c_l(E^{D_2} /\mathbb{Q})  \prod_{q |N_- } c_q(E^{D_2} /\mathbb{Q})   \nonumber 
\end{eqnarray}

\bigskip
But now for primes $l$ dividing $N^I_{1,+} =N^I_{2,+}$ and $N^I_{1,-} =N^I_{2,-}$, we have that $\chi_1(l) = \chi_2(l)$, and so $E^{D_1}/\mathbb{Q}_l$
 and $E^{D_2}/\mathbb{Q}_l$ are isomorphic over $\mathbb{Q}_l$. In particular we have $c_l(E^{D_1}/\mathbb{Q})=c_l(E^{D_2}/\mathbb{Q})$ for such $l$. Hence by taking the product of \eqref{e3.9} and \eqref{e3.10}, we have
 \begin{eqnarray}\label{e3.11}
  \prod_{l |N}  c_l(E^{D_1} /\mathbb{Q})  \prod_{l |N}  c_l(E^{D_2} /\mathbb{Q}) =   \prod_{q|N_-} c_q(E^{D_1}/\mathbb{Q}) \cdot   c_q(E^{D_2}/\mathbb{Q})   \mod{(\mathbb{Q}^{\times})^2}
 \end{eqnarray}
and so \eqref{e3.8} follows by applying \eqref{e3.7} for each prime $q|N_-$.
\end{proof}

\bigskip
\bigskip
We can now state Theorem~\ref{3.4}, which includes Theorem~\ref{1.1} as a special case. The notations are as before, and in particular, the pair $(\chi_1,\chi_2)$ satisfies condition (*). Here, $N_-$ is a squarefree product of an odd number of distinct primes. The elliptic curve $E^{D_1}$ has conductor equal to $N \cdot D_1^2$, and we have the $L$-functions $L(s,E^{D_1}/\mathbb{Q})=L(s,E/\mathbb{Q},\chi_1)$, with sign of the functional equation $\epsilon(E^{D_1}/\mathbb{Q}) =\epsilon(E/\mathbb{Q},\chi_1)$. Similarly the elliptic curve $E^{D_2}$ has conductor equal to $N \cdot D_2^2$, and we have the $L$-functions $L(s,E^{D_2}/\mathbb{Q})=L(s,E/\mathbb{Q},\chi_2)$, with sign of the functional equation $\epsilon(E^{D_2}/\mathbb{Q}) =\epsilon(E/\mathbb{Q},\chi_2)$. Then as we have seen in the previous section, we have $ \epsilon(E/\mathbb{Q},\chi_1) = - \epsilon(E/\mathbb{Q},\chi_2)$. Thus the analytic ranks of $E^{D_1}$ and $E^{D_2}$ are of opposite parity. We are interested in the cases where the analytic rank of $E^{D_1}/\mathbb{Q}$ is equal to zero (resp. one), while the analytic rank of $E^{D_2}/\mathbb{Q}$ is equal to one (resp. zero).

\bigskip

The main theorem of this section is the following, which includes Theorem~\ref{1.1} as a special case (by taking $D_1=1$ and $D_2=D$):

\bigskip

\begin{theorem}\label{3.4}
Assume that $E^{D_1}/\mathbb{Q}$ is of analytic rank zero (resp. one), and that $E^{D_2}/\mathbb{Q}$ is of analytic rank one (resp. zero). Then the Birch and Swinnerton-Dyer formula modulo square of rational numbers is valid for $E^{D_1}/\mathbb{Q}$, if and only if it is valid for $E^{D_2}/\mathbb{Q}$.   
\end{theorem}

\begin{proof}
For the proof of Theorem~\ref{3.4}, by symmetry we may assume without loss of generality that $E^{D_1}/\mathbb{Q}$ is of analytic rank zero, while $E^{D_2}/\mathbb{Q}$ is of analytic rank one, i.e. that $L(1,E/\mathbb{Q},\chi_1) \neq 0$ and $L^{\prime}(1,E/\mathbb{Q},\chi_2)\neq 0$. By Kolyvagin \cite{K}, the elliptic curves $E^{D_1}/\mathbb{Q}$ and $E^{D_2}/\mathbb{Q}$ satisfy the weak form of the Birch and Swinnerton-Dyer conjecture, and their Shafarevich-Tate groups are both finite, and hence their orders are squares.

\bigskip

The Birch and Swinnerton-Dyer formula modulo square of rational numbers for $E^{D_1}/\mathbb{Q}$ is then the validity of the formula:

\begin{eqnarray}\label{e3.12}
& & L(1,E^{D_1}/\mathbb{Q}) \\
&=& \prod_{l | N D_1} c_l(E^{D_1}/\mathbb{Q}) \cdot \Omega^+_{E^{D_1}/\mathbb{Q}}
\mod{(\mathbb{Q}^{\times})^2} \nonumber \\
&=& \prod_{l | N} c_l(E^{D_1}/\mathbb{Q}) \cdot \prod_{l | D_1} c_l(E^{D_1}/\mathbb{Q}) \cdot \Omega^+_{E^{D_1}/\mathbb{Q}} \mod{(\mathbb{Q}^{\times})^2} \nonumber
\end{eqnarray}

\bigskip
While the Birch and Swinnerton-Dyer formula modulo square of rational numbers for $E^{D_2}/\mathbb{Q}$ is then the validity of the formula:
\begin{eqnarray}\label{e3.13}
& & L^{\prime}(1,E^{D_2}/\mathbb{Q}) \\
&=& \mathrm{Reg}(E^{D_2}/\mathbb{Q}) \cdot  \prod_{l | N D_2} c_l(E^{D_2}/\mathbb{Q}) \cdot \Omega^+_{E^{D_2}/\mathbb{Q}} \mod{(\mathbb{Q}^{\times})^2}  \nonumber \\
&=&  \mathrm{Reg}(E^{D_2}/\mathbb{Q}) \cdot  \prod_{l |  N} c_l(E^{D_2}/\mathbb{Q}) \cdot  \prod_{l |  D_2} c_q(E^D/\mathbb{Q}) \cdot \Omega^+_{E^{D_2}/\mathbb{Q}} \nonumber 
\mod{(\mathbb{Q}^{\times})^2}
\end{eqnarray}

\bigskip
With $\chi_F$ being the genus character of $F$ associated to the pair $(\chi_1,\chi_2)$ as before, we have:
\begin{eqnarray}\label{e3.14}
L(s,E/F,\chi_F) &=&  L(s,E/\mathbb{Q},\chi_1) \cdot L(s,E/\mathbb{Q},\chi_2) \\
&=& L(s,E^{D_1}/\mathbb{Q}) \cdot L(s,E^{D_2}/\mathbb{Q}) \nonumber
\end{eqnarray}
so in particular 
\begin{eqnarray}\label{e3.15}
L^{\prime}(1,E/F,\chi_F) =  L(1,E^{D_1}/\mathbb{Q}) \cdot L^{\prime}(1,E^{D_2}/\mathbb{Q}) \neq 0
\end{eqnarray}

\bigskip
Now, by using Theorem~\ref{2.1} (and the discussion in the paragraph following the proof of Theorem~\ref{2.1}), we then obtain:
\begin{eqnarray}\label{e3.16}
& &  L^{\prime}(1,E/F,\chi_F) \\
&= & \frac{1}{\sqrt{D}}  \prod_{q|N_-} c_q(E/F) \cdot \mathrm{height}_F(\mathbf{P}) \cdot (\Omega^+_{E/\mathbb{Q}})^2 \mod{(\mathbb{Q}^{\times})^2} \nonumber \\
&=& \frac{2}{\sqrt{D_1} \sqrt{D_2}}  \prod_{q|N_-} c_q(E/F) \cdot \mathrm{height}_{\mathbb{Q}}(\mathbf{P}) \cdot (\Omega^+_{E/\mathbb{Q}})^2 \mod{(\mathbb{Q}^{\times})^2} \nonumber
\end{eqnarray}
for some non-torsion $\mathbf{P}$ of the one dimensional $\mathbb{Q}$-vector space $E^{D_2}(\mathbb{Q}) \otimes_{\mathbb{Z}} \mathbb{Q}$.

\bigskip

On applying Theorem~\ref{3.1} (whose proof depends on Theorem~\ref{1.3}) and the definition of the $u_{D_i}$ ($i=1,2$), we then obtain (noting that $\omega(N_-)$ is odd):
\begin{eqnarray}\label{e3.17}
& &  L^{\prime}(1,E/F,\chi_F) \\
&=& \frac{u_{D_1}  u_{D_2}}{2}  \prod_{l |D_1} c_l(E^{D_1}/\mathbb{Q})  \prod_{l |D_2} c_l(E^{D_2}/\mathbb{Q}) \prod_{q |N_-} \widetilde{c}_q(E/\mathbb{Q}) \nonumber \\
& & \cdot \,\ \frac{2}{\sqrt{D_1} \sqrt{D_2}}  \prod_{q|N_-} c_q(E/F) \times \mathrm{height}_{\mathbb{Q}}(\mathbf{P}) \cdot (\Omega^+_{E/\mathbb{Q}})^2 \mod{(\mathbb{Q}^{\times})^2} \nonumber \\
&=& \left( \prod_{l |D_1} c_l(E^{D_1}/\mathbb{Q})  \prod_{l |D_2} c_l(E^{D_2}/\mathbb{Q}) \right) \cdot \left( \prod_{q |N_-} c_q(E/F) \widetilde{c}_q(E/\mathbb{Q})  \right) \nonumber \\
& & \times  \,\  \mathrm{height}_{\mathbb{Q}}(\mathbf{P})  \cdot \Omega^+_{E^{D_1}/\mathbb{Q}} \cdot \Omega^+_{E^{D_2}/\mathbb{Q}} \mod{(\mathbb{Q}^{\times})^2}  \nonumber 
\end{eqnarray}
and combining Lemma~\ref{3.3} with \eqref{e3.17}, we then obtain
\begin{eqnarray}\label{e3.18}
& &  L^{\prime}(1,E/F,\chi_F) \\
&=&  \left( \prod_{l |D_1} c_l(E^{D_1}/\mathbb{Q})  \prod_{l |D_2} c_l(E^{D_2}/\mathbb{Q}) \right) \cdot \left( \prod_{l |N} c_l(E^{D_1}/\mathbb{Q})  \prod_{l |N} c_l(E^{D_2}/\mathbb{Q})  \right) \nonumber  \nonumber \\
& & \times \,\  \mathrm{height}_{\mathbb{Q}}(\mathbf{P})  \cdot \Omega^+_{E^{D_1}/\mathbb{Q}} \cdot \Omega^+_{E^{D_2}/\mathbb{Q}} \mod{(\mathbb{Q}^{\times})^2}  \nonumber  
\end{eqnarray}

\bigskip

Finally if $\mathbb{P}^D$ is a generator of $E^{D_2}(\mathbb{Q})$ modulo torsion, then in $E^{D_2}(\mathbb{Q}) \otimes_{\mathbb{Z}} \mathbb{Q}$ we have:
\begin{eqnarray}\label{e3.19}
\mathbf{P} = r \cdot \mathbb{P}^D    
\end{eqnarray}
for some non-zero rational number $r$; hence we have:
\begin{eqnarray}\label{e3.20}
 \mathrm{Reg}(E^D/\mathbb{Q}) &=& \mathrm{height}_{\mathbb{Q}}(\mathbb{P}^D) \\
& = &  \mathrm{height}_{\mathbb{Q}}(\mathbf{P}) \mod{(\mathbb{Q}^{\times})^2} \nonumber
\end{eqnarray}

\bigskip

Thus on combining \eqref{e3.15}, \eqref{e3.18}, and \eqref{e3.20}, we then see that \eqref{e3.12} is valid if and only if \eqref{e3.13} is valid. This completes the proof of Theorem~\ref{3.4}, modulo the proof of Theorem~\ref{1.3} (to be completed in Section~\ref{Sec4}). 
\end{proof}

\section{Proof of Theorem~\ref{1.3}}\label{Sec4}

To complete the proof of Theorem~\ref{1.1} and Theorem~\ref{3.4}, it remains to establish Theorem~\ref{1.3}. We thus consider an elliptic curve $E/\mathbb{Q}$ whose conductor $N$ is written in the form $N=N_+ N_-$, with $N_+,N_-$ being relatively prime, and $N_-$ being squarefree (we consider both the cases where $N_-$ is a squarefree product of an odd or even number of distinct primes; in particular $N_-$ is allowed to be equal to one). Again $D$ is a positive fundamental discriminant that satisfies the modified Heegner hypothesis with respect to $(N_+,N_-)$: all primes dividing $N_+$ split in $\mathbb{Q}(\sqrt{D})$, and all primes dividing $N_-$ are inert in $\mathbb{Q}(\sqrt{D})$. The rest of the notations are the same as before, but no conditions on the analytic ranks are imposed on $E/\mathbb{Q}$ or $E^D/\mathbb{Q}$. We need to show that the quantity \eqref{e1.2} is an even power of two. There are serious technical difficulties in the case when $D$ is even, because there is no simple description of local Tamagawa numbers at the prime two in the case of additive reduction, and our proof crucially relies on the recent classification \cite{BRSTTW} of how local Tamagawa numbers change under quadratic twists.  

\bigskip

Throughout this section, we employ the standard notation for affine Weierstrass equations of elliptic curves. To compute the local Tamagawa number of an elliptic curve $E$ at a prime $p$, one needs to consider an integral model for $E$ that is minimal at the prime $p$. That is, the model's discriminant has $p$-adic valuation equal to the $p$-adic valuation of the minimal discriminant of $E$. To ease notation, we will identify an elliptic curve $E$ with a given Weierstrass model. Thus, we recall that if $E$ and $E'$ are elliptic curves defined over a field~$K$, then a $K$-isomorphism $\phi:E\to E'$ has the form $\phi(x,y)=(u^2 x + r,u^3x+u^2sx+w)$ for some $u,r,s,w \in K$ with $u\not =0$. In what follows, we identify $\phi$ with $[u,r,s,w]$. 

\bigskip

We now consider the following lemma, which will provide us with a model to identify the quadratic twist $E^d$ by a general non-zero integer $d$ (here $d$ is not required to be a fundamental discriminant).    

\begin{lemma}\label{4.1}
\label{Lem:QuadTwist}Suppose an elliptic curve $E$ is given by an affine Weierstrass model
\[
y^{2}+a_{1}xy+a_{3}y=x^{3}+a_{2}x^{2}+a_{4}x+a_{6}.
\]
Then, for a non-zero integer $d$, the quadratic twist of $E$ by $d$ has an affine Weierstrass model
\begin{align*}
E^{d}  & :y^{2}+a_{1}xy+a_{3}y\\
& =x^{3}+\left(  a_{2}d+a_{1}^{2}\frac{d-1}{4}\right)  x^{2}+\left(
a_{4}d^{2}+a_{1}a_{3}\frac{d^{2}-1}{2}\right)  x+ \left(a_{6}d^{3}+a_{3}^{2} %
\frac{d^{3}-1}{4} \right).
\end{align*}
Moreover, if $d=s^{2}f$ for some integers $s$ and $f$, then $\left[  s,0,\frac{a_{1}\left(  s-1\right)  }{2},\frac{a_{3}(s^{3}-1)}{2}\right]  $ is a $\mathbb{Q}$-isomorphism of elliptic curves (in terms of affine Weierstrass equations) from $E^{d}$ to $E^{f}$.
\end{lemma}

\begin{proof}
That $E^{d}$ is an affine Weierstrass model for the quadratic twist of $E$ by $d$ is given in Proposition 4.3.2 of \cite{connell}. It is easily verified that $\left[  s,0,\frac{a_{1}\left(  s-1\right)  }{2},\frac{a_{3}(s^{3}-1)}{2}\right]  $ is a $\mathbb{Q}$-isomorphism of elliptic curves from $E^{d}$ to $E^{f}$ by using software such as Sage \cite{sagemath}.
\end{proof}

\bigskip

The following lemma determines the value of $u_D$ that appears in \eqref{e1.2}. This quantity is the positive integer $u$ appearing in a $\mathbb{Q}$-isomorphism between $E^D$ and any global minimal model for $E^D$. 

\bigskip

\begin{lemma}\label{lem:uvalue}
Suppose $E/\mathbb{Q}$ is an elliptic curve given by a global minimal model, and let $c_6$ denotes its associated invariant. Let $D$ be a fundamental discriminant that is coprime to the conductor of $E$ and let $E^{D}$ be the quadratic twist of $E$ by $D$. If $\left[u_{D},s_{D},r_{D},w_{D}\right]  $ is a $\mathbb{Q}$-isomorphism from $E^{D}$ onto a global minimal model for $E^{D}$, then
\begin{equation}\label{eq:uval}
u_{D}=\left\{
\begin{array}
[c]{cl}
1 & \text{if }v_{2}(D)\leq2\text{ or }v_{2}(D)=3\text{ with }v_{2}(c_{6})=0,\\
2 & \text{if }v_{2}(D)=3\text{ with }v_{2}(c_{6})=3.
\end{array}
\right.
\end{equation}
\end{lemma}

\begin{proof}
Suppose $D$ is odd so that $D\equiv1\ \operatorname{mod}4$ is squarefree. By~\cite[Corollary 2.6]{Pal}, $u_{D}=1$. Now suppose that $D$ is even and write $D=4m$ for some squarefree integer $m$. Since $D$ is coprime to the conductor of $E$, it follows that $E$ has good reduction at $2$. As $E$ is given by a global minimal model, \cite[Tableau IV]{Papadopoulos1993} implies that $v_{2}(c_{6})\in\left\{  0,3\right\}  $. By Lemma \ref{Lem:QuadTwist}, $E^{D}$ and $E^{m}$ are $\mathbb{Q}$-isomorphic. Now let $\left[  \widetilde{u}_{D},\widetilde{r}_{D},\widetilde{s}_{D},\widetilde{w}_{D}\right]  $ be a $\mathbb{Q}$-isomorphism from $E^{m}$ onto a global minimal model for $E^{D}$. Then, \cite[Proposition 2.5]{Pal} implies that
\[
\widetilde{u}_{D}=\left\{
\begin{array}
[c]{cl}
\frac{1}{2} & \text{if }m\equiv3\ \operatorname{mod}4\text{ or }m\equiv2\ \operatorname{mod}4\text{ with }v_{2}(c_{6})=0,\\
1 & \text{if }m\equiv2\ \operatorname{mod}4\text{ with }v(c_{6})=3.
\end{array}
\right.
\]
Now let $\left[  u_{D}^{\prime},s_{D}^{\prime},r_{D}^{\prime},w_{D}^{\prime}\right]  $ be a $\mathbb{Q}$-isomorphism from $E^{D}$ onto $E^{m}$. By Lemma \ref{Lem:QuadTwist}, $u_{D}^{\prime}=2$. The result now follows since if $\left[  u_{D},s_{D},r_{D},w_{D}\right]  $ is a $\mathbb{Q}$-isomorphism from $E^{D}$ onto a global minimal model for $E^{D}$, then it is the case that $u_{D}=u_{D}^{\prime}\widetilde{u}_{D}=2\widetilde{u}_{D}$. 
\end{proof}

\bigskip

We next record the following lemma, which provides an equivalence between a statement involving Kronecker symbols and one of the quantities appearing in Theorem~\ref{1.3} being a square.

\bigskip

\begin{lemma}
\label{Lem:TamRecRel} Let $E/\mathbb{Q}$ be an elliptic curve with conductor $N$. Let $D=2^{a}m$ for $a$ a nonnegative integer and $m$ an odd positive squarefree integer that is coprime to $N$. If $E^{m}$ denotes the quadratic twist of $E$ by $D$, then
\[
\prod_{l|m}c_{l}(E^{D}/\mathbb{Q})\text{ is a square}\qquad
\Longleftrightarrow\qquad\left(  \frac{\Delta}{m}\right)  =\prod_{l
|m}\left(  \frac{\Delta}{l}\right)  =1.
\]
Moreover, $\prod_{l|m}c_{l}(E^{D}/\mathbb{Q})$ is a power of $2$.
\end{lemma}

\begin{proof}
Let $E$ be given by a global minimal model $y^{2}+a_{1}xy+a_{3}y=x^{3}+a_{2}x^{2}+a_{4}x+a_{6}$ with minimal discriminant $\Delta$. Then, the admissible change of variables $\left[  1,0,\frac{-a_{1}}{2},\frac{-a_{3}}{2}\right]  $ results in a $\mathbb{Q}$-isomorphic affine Weierstrass model $y^{2}=f(x)$, where $f(x)\in\mathbb{Z}\!\left[  \frac{1}{2}\right]  \left[  x\right]  $ is a cubic polynomial. In particular, $y^{2}=f(x)$ is a minimal model for each odd prime.

\bigskip

Now suppose that $l$ is an odd prime dividing $D$. Equivalently, $l|m$. Since $N$ and $D$ are coprime, $E$ has good reduction at $l$ and $E^{D}$ has additive reduction at $l$. By \cite[Proposition~1]{Comalada} and \cite[Proposition~5]{Rubin07}, we have that $\operatorname*{typ}_{l}\!\left(  E^{D}\right)  =\mathrm{{I}_{0}^{\ast}}$ and
\[
c_{l}\!\left(  E^{D}/\mathbb{Q}\right)  =1+\#\left\{  x\in\mathbb{F}_{l}\mid f(x)\equiv
0\ \operatorname{mod}l\right\}  \in\left\{  1,2,4\right\}  .
\]
By construction, we have that $\Delta=16\operatorname*{disc}(f)$, where $\operatorname*{disc}(f)$ denotes the discriminant of $f$. Since $E$ has good reduction at $l$, we have that $l$ is unramified in $K=\mathbb{Q}\!\left(  \sqrt{\Delta}\right)  $. Thus,
\begin{align*}
\#\left\{  x\in\mathbb{F}_{l}\mid f(x)\equiv0\ \operatorname{mod}l\right\}    & =\left\{
\begin{array}
[c]{cl}
0,3 & \text{if }\left(  \frac{\Delta}{l}\right)  =1,\\
1 & \text{if }\left(  \frac{\Delta}{l}\right)  =-1,
\end{array}
\right.  \\
\Longrightarrow\qquad c_{l}(E^{D}/\mathbb{Q})  & =\left\{
\begin{array}
[c]{cl}
1,4 & \text{if }\left(  \frac{\Delta}{l}\right)  =1,\\
2 & \text{if }\left(  \frac{\Delta}{l}\right)  =-1.
\end{array}
\right.
\end{align*}
It now follows that $\prod_{l|m}(c_{l}E^{D}/\mathbb{Q})$ is a power of $2$ and $\prod_{l|m}c_{l}(E^{D}/\mathbb{Q})$ is a square if and only if $\left(  \frac{\Delta}{m}\right)  =\prod_{l|m}\left(  \frac{\Delta}{l}\right)  =1$.
\end{proof}

\bigskip

Our following result provides us with explicit conditions to determine the Kronecker symbol~$\left(\frac{\Delta}{m}\right) $ in Lemma~\ref{Lem:TamRecRel}.

\bigskip

\begin{lemma}
\label{Lem:QuadRec}Let $E/\mathbb{Q}$ be an elliptic curve with minimal discriminant $\Delta$ and conductor $N=N_{+}N_{-}$, where $N_{+}$ and $N_{-}$ are coprime with $N_{-}$ squarefree. Let $D$ be a positive fundamental discriminant that satisfies the modified Heegner hypothesis with respect to $\left(  N_{+},N_{-}\right)  $. In particular, $D=2^{a}m$ for $a\in\{0,2,3\}$ and $m$ a squarefree positive odd integer. Set $b=\sum_{p|N_{-}}v_{p}(\Delta)$. Then,

\begin{enumerate}
\item if $v_{2}(D)=0$, then $\left(  \frac{\Delta}{D}\right)  =\left(\frac{\Delta}{m}\right)  =\left(  -1\right)  ^{b}$;

\item if $v_{2}(D)=2$, then%
\[
\left(  \frac{\Delta}{m}\right)  =\left\{
\begin{array}
[c]{cl}%
\left(  -1\right)  ^{b} & \text{if }\Delta\equiv1\ \operatorname{mod}4,\\
\left(  -1\right)  ^{b+1} & \text{if }\Delta\equiv3\ \operatorname{mod}4.
\end{array}
\right.
\]

\item if $v_{2}(D)=3$, then%
\[
\left(  \frac{\Delta}{m}\right)  =\left\{
\begin{array}
[c]{cl}
\left(  -1\right)  ^{b} & \text{if }\left(  i\right)  \ \Delta\equiv1 \ \operatorname{mod}8\text{ or }\left(  ii\right)  \ \Delta\equiv 3\ \operatorname{mod}8\text{ and }m\equiv3\ \operatorname{mod}4\text{ or}\ \\
& \qquad\left(  iii\right)  \ \Delta\equiv7\ \operatorname{mod}8\text{ and }m\equiv1\ \operatorname{mod}4,\\
\left(  -1\right)  ^{b+1} & \text{if }\left(  i\right)  \ \Delta\equiv5 \ \operatorname{mod}8\text{ or }\left(  ii\right)  \ \Delta\equiv 3\ \operatorname{mod}8\text{ and }m\equiv1\ \operatorname{mod}4\text{ or}\ \\
& \qquad\left(  iii\right)  \ \Delta\equiv7\ \operatorname{mod}8\text{ and }m\equiv3\ \operatorname{mod}4.
\end{array}
\right.
\]

\end{enumerate}
\end{lemma}

\begin{proof}
Let $F=\mathbb{Q}(\sqrt{D})$ and write $\Delta=\pm\Delta_{+}\Delta_{-}$ where
\[
\Delta_{+}=\prod_{l|N_{+}}l^{v_{l}(\Delta)}\qquad\text{and}\qquad \Delta_{-}=\prod_{p|N_{-}}p^{v_{p}(\Delta)}.
\]
By assumption, $\Delta_{+}$ and $\Delta_{-}$ are coprime. Since $D$ satisfies the modified Heegner hypothesis with respect to $\left(  N_{+},N_{-}\right)  $ we have that each prime $l$ dividing $\Delta_{+}$ splits in $F$. Thus, we have the the Kronecker symbol $\left(  \frac{D}{l}\right)  =1$. Similarly, since each prime $p$ dividing $\Delta_{-}$ is inert in $F$ we have that the Kronecker symbol $\left(  \frac{D}{p}\right)  =-1$. Then, with $b=\sum_{p|N_{-}}v_{p}(\Delta)=\sum_{p|\Delta_{-}}v_{p}(\Delta_{-})$ we have that
\begin{align*}
\left(  \frac{D}{\Delta_{+}}\right)    & =\prod_{l|\Delta_{+}}\left(\frac{D}{l}\right)  =1,\\
\left(  \frac{D}{\Delta_{-}}\right)    & =\prod_{p|N_{-}}\left(  \frac{D}{p^{v_{p}(\Delta)}}\right)  =\prod_{p|N_{-}}\left(  -1\right)^{v_{p}(\Delta)}=\left(  -1\right)  ^{b}.
\end{align*}

\bigskip

Now suppose that $D$ is odd so that $D=m\equiv1\ \operatorname{mod}4$. Write $\Delta_{\pm}=2^{v_{2}(\Delta_{\pm})}\Gamma_{\pm}$ for some odd positive integers $\Gamma_{+}$ and $\Gamma_{-}$. Note that at most one of $\Delta_{+}$ or $\Delta_{-}$ could be even since $\Delta_{+}$ and $\Delta_{-}$ are coprime. Moreover, by assumption we have that if $2$ divides $\Delta_{+}$ (resp. $\Delta_{-}$), then $2$ splits (resp. is inert) in $F$. Consequently, $m\equiv 1\ \operatorname{mod}8$ (resp. $5\ \operatorname{mod}8$) and we have the Kronecker symbol
\[
\left(  \frac{2}{m}\right)  =\left\{
\begin{array}
[c]{cl}%
1 & \text{if }v_{2}(\Delta_{+})>0,\\
-1 & \text{if }v_{2}(\Delta_{-})>0.
\end{array}
\right.
\]
Hence, by properties of the Kronecker symbol and quadratic reciprocity, we have that
\begin{align*}
\left(  \frac{\Delta}{m}\right)  =\left(  \frac{\left\vert \Delta\right\vert}{m}\right)    & =\left(  \frac{2}{m}\right)  ^{v_{2}(\Delta_{+})}\left(\frac{2}{m}\right)  ^{v_{2}(\Delta_{-})}\left(  \frac{\Gamma_{+}\Gamma_{-}}{m}\right)  \\
& =\left(  -1\right)  ^{v_{2}(\Delta_{-})}\left(  \frac{m}{\Gamma_{+}\Gamma_{-}}\right)  \\
& =\left(  -1\right)  ^{v_{2}(\Delta_{-})}\left(  \frac{m}{\Gamma_{-}}\right)\left(  \frac{m}{\Gamma_{+}}\right)  \\
& =\left(  -1\right)  ^{v_{2}(\Delta_{-})}\prod_{l|\Gamma_{-}}\left(\frac{m}{l}\right)  ^{v_{l}(\Gamma_{-})}\prod_{p|\Gamma_{+}}\left(\frac{m}{p}\right)  ^{v_{p}(\Delta_{+})}\\
& =\prod_{l|\Delta_{-}}\left(  -1\right)  ^{v_{l}(\Delta_{-})}\\
& =\left(  -1\right)  ^{b}.
\end{align*}

\bigskip

Next, we consider the case when $D$ is even. Thus, $D=2^{a}m$ with $a\in\left\{  2,3\right\}  $ and $\Delta_{+}\Delta_{-}$ is an odd positive integer. Moreover,
\begin{equation}
1=\left(  \frac{D}{\Delta_{+}}\right)  =\left(  \frac{2}{\Delta_{+}}\right)^{a}\left(  \frac{m}{\Delta_{+}}\right)  \qquad\Longrightarrow\qquad\left(\frac{m}{\Delta_{+}}\right)  =\left\{
\begin{array}
[c]{cl}
1 & \text{if }v_{2}(D)=2,\\
\left(  \frac{2}{\Delta_{+}}\right)   & \text{if }v_{2}(D)=3.
\end{array}
\right.  \label{quadrecDDplus}%
\end{equation}
Since each odd number is a square modulo $8$, we obtain
\[
\left(  \frac{\Delta}{D}\right)  =\left(  \frac{\Delta}{2^{a}}\right)  \left(\frac{\Delta}{m}\right)  =\left(  \frac{\Delta}{m}\right)  .
\]
In addition, we have that%
\[
\left(  -1\right)  ^{b}=\left(  \frac{D}{\Delta_{-}}\right)  =\left(\frac{2^{a}}{\Delta_{-}}\right)  \left(  \frac{m}{\Delta_{-}}\right)  .
\]
Since $m$ and $\Delta_{+}\Delta_{-}$ are positive coprime odd integers, we have by quadratic reciprocity and properties of the Kronecker symbol that
{\footnotesize \begin{align*}
\left(  \frac{\Delta}{m}\right)  =\left(  \frac{\pm1}{m}\right)  \left(\frac{\Delta_{+}\Delta_{-}}{m}\right)    & =\left\{
\begin{array}
[c]{cl}
\left(  \frac{\pm1}{m}\right)  \left(  \frac{m}{\Delta_{+}\Delta_{-}}\right) & \text{if }\Delta_{+}\Delta_{-}\equiv1\ \operatorname{mod}4\text{ or } m\equiv1\ \operatorname{mod}4,\\
-\left(  \frac{\pm1}{m}\right)  \left(  \frac{m}{\Delta_{+}\Delta_{-}}\right) & \text{if }\Delta_{+}\Delta_{-},m\equiv3\ \operatorname{mod}4,
\end{array}
\right.  \\
& =\left\{
\begin{array}
[c]{cl}%
\left(  \frac{\pm1}{m}\right)  \left(  \frac{m}{\Delta_{+}}\right)  \left(\frac{2^{a}}{\Delta_{-}}\right)  \left(  \frac{2^{a}}{\Delta_{-}}\right) \left(  \frac{m}{\Delta_{-}}\right)   & \text{if }\Delta_{+}\Delta_{-}\equiv1\ \operatorname{mod}4\text{ or }m\equiv1\ \operatorname{mod}4,\\
-\left(  \frac{\pm1}{m}\right)  \left(  \frac{m}{\Delta_{+}}\right)  \left(\frac{2^{a}}{\Delta_{-}}\right)  \left(  \frac{2^{a}}{\Delta_{-}}\right) \left(  \frac{m}{\Delta_{-}}\right)   & \text{if }\Delta_{+}\Delta_{-} ,m\equiv3\ \operatorname{mod}4,
\end{array}
\right.  \\
& =\left\{
\begin{array}
[c]{cl}
\left(  \frac{\pm1}{m}\right)  \left(  \frac{m}{\Delta_{+}}\right)  \left(\frac{2^{a}}{\Delta_{-}}\right)  \left(  \frac{D}{\Delta_{-}}\right)   & \text{if }\Delta_{+}\Delta_{-}\equiv1\ \operatorname{mod}4\text{ or } m\equiv1\ \operatorname{mod}4,\\
-\left(  \frac{\pm1}{m}\right)  \left(  \frac{m}{\Delta_{+}}\right)  \left(\frac{2^{a}}{\Delta_{-}}\right)  \left(  \frac{D}{\Delta_{-}}\right)   & \text{if }\Delta_{+}\Delta_{-},m\equiv3\ \operatorname{mod}4,
\end{array}
\right.  \\
& =\left\{
\begin{array}
[c]{cl}%
\left(  \frac{\pm1}{m}\right)  \left(  \frac{m}{\Delta_{+}}\right)  \left( \frac{2^{a}}{\Delta_{-}}\right)  \left(  -1\right)  ^{b} & \text{if } \Delta_{+}\Delta_{-}\equiv1\ \operatorname{mod}4\text{ or }m\equiv 1\ \operatorname{mod}4,\\
-\left(  \frac{\pm1}{m}\right)  \left(  \frac{m}{\Delta_{+}}\right)  \left(\frac{2^{a}}{\Delta_{-}}\right)  \left(  -1\right)  ^{b} & \text{if }\Delta_{+}\Delta_{-},m\equiv3\ \operatorname{mod}4.
\end{array}
\right.
\end{align*}}
We now consider the cases $v_{2}(D)=2$ and $v_{2}(D)=3$ separately.

\bigskip

\textbf{Case 1.} Suppose $v_{2}(D)=2$. Thus, $m\equiv3\ \operatorname{mod}4$ as $D$ is a positive fundamental discriminant. Then, $\left(  \frac{2^{a}}{\Delta_{-}}\right)  =\left(  \frac{m}{\Delta_{+}}\right)  =1$, where the second equality comes from (\ref{quadrecDDplus}). Therefore,
\begin{align*}
\left(  \frac{\Delta}{m}\right)    & =\left\{
\begin{array}
[c]{cl}%
\left(  \frac{\pm1}{m}\right)  \left(  -1\right)  ^{b} & \text{if }\Delta
_{+}\Delta_{-}\equiv1\ \operatorname{mod}4\text{ or }m\equiv
1\ \operatorname{mod}4,\\
-\left(  \frac{\pm1}{m}\right)  \left(  -1\right)  ^{b} & \text{if }\Delta
_{+}\Delta_{-},m\equiv3\ \operatorname{mod}4,
\end{array}
\right.  \\
& =\left\{
\begin{array}
[c]{cl}%
\left(  -1\right)  ^{b} & \text{if }\Delta>0\text{ and }\Delta_{+}\Delta_{-}\equiv1\ \operatorname{mod}4,\\
\left(  -1\right)  ^{b+1} & \text{if }\Delta<0\text{ and }\Delta_{+}\Delta_{-}\equiv1\ \operatorname{mod}4,\\
\left(  -1\right)  ^{b+1} & \text{if }\Delta>0\text{ and }\Delta_{+}\Delta_{-}\equiv3\ \operatorname{mod}4,\\
\left(  -1\right)  ^{b} & \text{if }\Delta<0\text{ and }\Delta_{+}\Delta_{-}\equiv3\ \operatorname{mod}4,
\end{array}
\right.  \\
& =\left\{
\begin{array}
[c]{cl}%
\left(  -1\right)  ^{b} & \text{if }\Delta\equiv1\ \operatorname{mod}4,\\
\left(  -1\right)  ^{b+1} & \text{if }\Delta\equiv3\ \operatorname{mod}4.
\end{array}
\right.
\end{align*}

\bigskip

\textbf{Case 2.} Suppose $v_{2}(D)=3$. Then, $\left(  \frac{2^{a}}{\Delta_{-}}\right)  =\left(  \frac{2}{\Delta_{-}}\right)  $ and $\left(  \frac{m}{\Delta_{+}}\right)  =\left(  \frac{2}{\Delta_{+}}\right)  $ by (\ref{quadrecDDplus}). Since $\Delta=\pm\Delta_{+}\Delta_{-}$, we have by properties of the Kronecker symbol that $\left(  \frac{2}{\Delta_{+}\Delta_{-}}\right)  =\left(  \frac{2}{\left\vert \Delta\right\vert }\right)  $. Now observe that
\begin{align*}
\left(  \frac{\Delta}{m}\right)    & =\left\{
\begin{array}
[c]{cl}%
\left(  \frac{\pm1}{m}\right)  \left(  \frac{2}{\Delta_{+}}\right)  \left(\frac{2}{\Delta_{-}}\right)  \left(  -1\right)  ^{b} & \text{if }\Delta_{+}\Delta_{-}\equiv1\ \operatorname{mod}4\text{ or }m\equiv 1\ \operatorname{mod}4,\\
-\left(  \frac{\pm1}{m}\right)  \left(  \frac{2}{\Delta_{+}}\right)  \left(\frac{2}{\Delta_{-}}\right)  \left(  -1\right)  ^{b} & \text{if }\Delta_{+}\Delta_{-},m\equiv3\ \operatorname{mod}4,
\end{array}
\right.  \\
& =\left\{
\begin{array}
[c]{cl}
\left(  \frac{\pm1}{m}\right)  \left(  \frac{2}{\left\vert \Delta\right\vert}\right)  \left(  -1\right)  ^{b} & \text{if }\Delta_{+}\Delta_{-}\equiv1\ \operatorname{mod}4\text{ or }m\equiv1\ \operatorname{mod}4,\\
-\left(  \frac{\pm1}{m}\right)  \left(  \frac{2}{\left\vert \Delta\right\vert}\right)  \left(  -1\right)  ^{b} & \text{if }\Delta_{+}\Delta_{-},m\equiv3\ \operatorname{mod}4.
\end{array}
\right.
\end{align*}

\bigskip

\qquad\textbf{Subcase 2a.} Suppose $\Delta=\pm\Delta_{+}\Delta_{-}\equiv1\ \operatorname{mod}8$ so that $\left(  \frac{2}{\left\vert \Delta\right\vert }\right)  =1$ and $\Delta_{+}\Delta_{-}\equiv\pm 1\ \operatorname{mod}4$. Then, if $\Delta>0$, we have that $\left( \frac{\Delta}{m}\right)  =\left(  -1\right)  ^{b}$. If $\Delta<0$, then $\Delta_{+}\Delta_{-}\equiv3\ \operatorname{mod}4$ and
\begin{align*}
\left(  \frac{\Delta}{m}\right)   &  =\left\{
\begin{array}
[c]{cl}
\left(  \frac{-1}{m}\right)  \left(  -1\right)  ^{b} & \text{if } m\equiv1\ \operatorname{mod}4,\\
-\left(  \frac{-1}{m}\right)  \left(  -1\right)  ^{b} & \text{if } m\equiv3\ \operatorname{mod}4,
\end{array}
\right.  \\
&  =\left(  -1\right)  ^{b}.
\end{align*}

\bigskip

\qquad\textbf{Subcase 2b.} Suppose $\Delta=\pm\Delta_{+}\Delta_{-}\equiv3\ \operatorname{mod}8$ so that $\left(  \frac{2}{\left\vert \Delta\right\vert }\right)  =-1$ and $\Delta_{+}\Delta_{-}\equiv \mp1\ \operatorname{mod}4$. If $\Delta>0$, then $\Delta_{+}\Delta_{-} \equiv3\ \operatorname{mod}4$ and
\[
\left(  \frac{\Delta}{m}\right)  =\left\{
\begin{array}
[c]{cl}
\left(  -1\right)  ^{b+1} & \text{if }m\equiv1\ \operatorname{mod}4,\\
\left(  -1\right)  ^{b} & \text{if }m\equiv3\ \operatorname{mod}4.
\end{array}
\right.
\]
Now suppose $\Delta<0$ so that $\Delta_{+}\Delta_{-}\equiv 1\ \operatorname{mod}4$. Then,
\[
\left(  \frac{\Delta}{m}\right)  =\left(  \frac{-1}{m}\right)  \left(
-1\right)  ^{b+1}=\left\{
\begin{array}
[c]{cl}
\left(  -1\right)  ^{b+1} & \text{if }m\equiv1\ \operatorname{mod}4,\\
\left(  -1\right)  ^{b} & \text{if }m\equiv3\ \operatorname{mod}4.
\end{array}
\right.
\]

\bigskip

\qquad\textbf{Subcase 2c.} Suppose $\Delta=\pm\Delta_{+}\Delta_{-} \equiv5\ \operatorname{mod}8$ so that $\left(  \frac{2}{\left\vert \Delta\right\vert }\right)  =-1$ and $\Delta_{+}\Delta_{-}\equiv \pm1\ \operatorname{mod}4$. If $\Delta>0$, then $\Delta_{+}\Delta_{-} \equiv1\ \operatorname{mod}4$ and thus $\left(  \frac{\Delta}{m}\right) =\left(  -1\right)  ^{b+1}$. If $\Delta<0$, then $\Delta_{+}\Delta_{-} \equiv3\ \operatorname{mod}4$ and thus
\begin{align*}
\left(  \frac{\Delta}{m}\right)   &  =\left\{
\begin{array}
[c]{cl}
\left(  \frac{-1}{m}\right)  \left(  -1\right)  ^{b+1} & \text{if } m\equiv1\ \operatorname{mod}4,\\
-\left(  \frac{-1}{m}\right)  \left(  -1\right)  ^{b+1} & \text{if } m\equiv3\ \operatorname{mod}4,
\end{array}
\right.  \\
&  =\left(  -1\right)  ^{b+1}.
\end{align*}

\bigskip

\qquad\textbf{Subcase 2d.} Suppose $\Delta=\pm\Delta_{+}\Delta_{-}\equiv7\ \operatorname{mod}8$ so that $\left(  \frac{2}{\left\vert \Delta\right\vert }\right)  =1$ and $\Delta_{+}\Delta_{-}\equiv\mp 1\ \operatorname{mod}4$. If $\Delta>0$, then $\Delta_{+}\Delta_{-}\equiv3\ \operatorname{mod}4$ and thus
\[
\left(  \frac{\Delta}{m}\right)  =\left\{
\begin{array}
[c]{cl}
\left(  -1\right)  ^{b} & \text{if }m\equiv1\ \operatorname{mod}4,\\
\left(  -1\right)  ^{b+1} & \text{if }m\equiv3\ \operatorname{mod}4.
\end{array}
\right.
\]
If $\Delta<0$, then $\Delta_{+}\Delta_{-}\equiv1\ \operatorname{mod}4$ which
gives
\[
\left(  \frac{\Delta}{m}\right)  =\left(  \frac{-1}{m}\right)  \left(
-1\right)  ^{b}=\left\{
\begin{array}
[c]{cl}%
\left(  -1\right)  ^{b} & \text{if }m\equiv1\ \operatorname{mod}4,\\
\left(  -1\right)  ^{b+1} & \text{if }m\equiv3\ \operatorname{mod}4.
\end{array}
\right.
\]
\end{proof}

\bigskip

We now have all the necessary ingredients to prove Theorem~\ref{1.3}:

\bigskip

\begin{proof}
[Proof of Theorem~\ref{1.3}]Let $\Delta=\pm\Delta_{+}\Delta_{-}$ where
\[
\Delta_{+}=\prod_{l|N_{+}}l^{v_{l}(\Delta)}\qquad\text{and}\qquad \Delta_{-}=\prod_{p|N_{-}}p^{v_{p}(\Delta)}.
\]
By assumption, $\Delta_{+}$ and $\Delta_{-}$ are coprime. Now set
\[
b=\sum_{p|N_{-}}v_{p}(\Delta)=\sum_{p|\Delta_{-}}v_{p}(\Delta_{-}).
\]

\bigskip

Next, recall that if $p|N_{-}$, then $\widetilde{c}_{p}(E/\mathbb{Q})=2-\left(  v_{p}(\Delta)\ \operatorname{mod}2\right)  =2^{1-(v_{p}(\Delta)\ \operatorname{mod}2)}$. Consequently,
\begin{equation}
\prod_{p|N_{-}}\widetilde{c}_{p}(E/\mathbb{Q})=\prod_{p|N_{-}}2^{1-(v_{p}(\Delta)\ \operatorname{mod}2)} =2^{\Sigma_{p|N_{-}}\left(  1-(v_{p}(\Delta)\ \operatorname{mod}2)\right)}.\label{cptil1}
\end{equation}
Moreover,
\begin{align}
\sum_{p|N_{-}}\left(  1-(v_{p}(\Delta)\ \operatorname{mod}2)\right)    &
=\omega(N_{-})-\sum_{p|N_{-}}(v_{p}(\Delta)\ \operatorname{mod}2)\nonumber\\
& \equiv\omega(N_{-})+b\ \operatorname{mod}2\nonumber\\
& \equiv\left\{
\begin{array}
[c]{cl}
1+b\ \operatorname{mod}2 & \text{if }\omega(N_{-})\text{ is odd,}\\
b\ \operatorname{mod}2 & \text{if }\omega(N_{-})\ \text{is even.}
\end{array}
\right.  \label{cptil2}%
\end{align}

\bigskip

Now suppose that $D$ is odd so that $D\equiv1\ \operatorname{mod}4$ as it is a positive fundamental discriminant. By Lemma \ref{lem:uvalue}, $u_{D}=1$. Then, (\ref{cptil1}) and \ref{cptil2} imply that
\begin{align}
& \frac{u_{D}}{2^{\omega(N_{-})}}\prod_{p|N_{-}}\widetilde{c}_{p}(E/\mathbb{Q})=2^{\Sigma_{p|N_{-}}\left(  1-(v_{p}(\Delta)\ \operatorname{mod}2)\right) -\omega(N_{-})}\label{cptil3}\\
& \Longrightarrow\qquad\sum_{p|N_{-}}\left(  1-(v_{p}(\Delta)\ \operatorname{mod}2)\right)  -\omega(N_{-}) \equiv b\ \operatorname{mod}2.\nonumber
\end{align}
It follows that (\ref{cptil3}) is a power of $2$ and it is an even power of $2$ if and only if $\left(  -1\right)  ^{b}=1$. On the other hand, Lemmas \ref{Lem:TamRecRel} and \ref{Lem:QuadRec} imply that
\[
\prod_{l|D}c_{l}(E^{D}/\mathbb{Q})
\]
is a power of $2$, and it is an even power of $2$ if and only if $\left(\frac{\Delta}{D}\right)  =\left(  \frac{\Delta}{m}\right)  =\left(  -1\right)^{b}=1$. It follows that \eqref{e1.2} is an even power of $2$.

\bigskip

It remains to consider the case when $D$ is even. To this end, write $D=2^{a}m$ for $a\in\left\{  2,3\right\}  $ and $m$ a positive odd integer. Note that if $a=2$, then $m\equiv3\ \operatorname{mod}4$ as $D$ is a fundamental discriminant. Since $D$ and the conductor $N$ of $E$ are coprime, we have that $E$ has good reduction at $2$. We may suppose that $E$ is given by a global minimal model that is also a $2$-strongly-minimal model \cite[Proposition 3.5]{BRSTTW}. In particular, $E$ is given by an affine Weierstrass model
\[
E:y^{2}+a_{1}xy+a_{3}y=x^{3}+a_{2}x^{2}+a_{4}x+a_{6}%
\]
where exactly one of the following is true:

\begin{enumerate}
\item $v_{2}(a_{1})=0,$ $v_{2}(a_{3})\geq2$, and either $\left(  i\right)  $ $v_{2}(a_{4})\geq1$ and $v_{2}(a_{6})=0$ or $\left(  ii\right)  \ v_{2}(a_{4})=0$ and $v_{2}(a_{6})\geq1$;

\item $v_{2}(a_{1}),v_{2}(a_{2})\geq1$ and $v_{2}(a_{3})=0$.
\end{enumerate}

\bigskip

Henceforth, we identify which $2$-strongly-minimal model of $E$ we are considering by the $2$-adic valuation of $a_{1}$. Next, let $c_6$ be the invariant associated to $E$. In particular, $c_6$ is the invariant associated to a global minimal model for $E$. Since
$$c_{6}=-\left(  a_{1}^{2}+4a_{2}\right)  ^{3}+36\left(  a_{1}^{2}%
+4a_{2}\right)  \left(  2a_{4}+a_{1}a_{3}\right)  -216\left(  a_{3}^{2}%
+4a_{6}\right),$$
it is verified from our assumptions that
\begin{equation}
v_{2}(c_{6})=\left\{
\begin{array}
[c]{cl}
0 & \text{if }v_{2}(a_{1})=0,\\
3 & \text{if }v_{2}(a_{1})\geq1.
\end{array}
\right.  \label{c6val}
\end{equation}

\bigskip

Next, we observe that since $E$ has good reduction at $2$, and $E^{D}$ has additive reduction at~$2$, then \cite[Theorem 5.1]{BRSTTW} implies that $c_{2}(E^{D}/\mathbb{Q})\in\left\{  1,2,4\right\}  $. From Lemma \ref{lem:uvalue} and (\ref{cptil1}) we have that
\[
\Lambda=\frac{u_{D}}{2^{\omega(N_{-})}}c_{2}(E^{D}/\mathbb{Q})\prod_{p|N_{-}}\widetilde{c}_{p}(E/\mathbb{Q})
\]
is a power of $2$. This paired with Lemma \ref{Lem:TamRecRel} and the fact that $u_{D}\prod_{p|N_{-}}\widetilde{c}_{p}(E/\mathbb{Q})\prod_{l|D}c_{l}(E^{D}/\mathbb{Q})$ is a positive integer leads us to conclude that
\[
\frac{u_{D}}{2^{\omega(N_{-})}}\prod_{p|N_{-}}\widetilde{c}_{p}(E/\mathbb{Q})\prod_{l|D}c_{l}(E^{D}/\mathbb{Q})=\Lambda\prod_{l|m}c_{l}(E^{D}/\mathbb{Q})=2^{k}
\]
for some integer $k$. In particular, to prove the theorem, it now suffices to show that $\Lambda\prod_{l|m}c_{l}(E^{D}/\mathbb{Q})$ is a square. We now show this by cases.

\bigskip

\textbf{Case 1.} Suppose $v_{2}(D)=2$ so that $m\equiv3\ \operatorname{mod}4$
is squarefree. By Lemma \ref{Lem:QuadTwist}, $E^{D}$ and $E^{m}$ are
$\mathbb{Q}$-isomorphic. By \cite[Theorem 5.1]{BRSTTW},
\[
\operatorname*{typ}(E^{D})=\operatorname*{typ}(E^{m})=\left\{
\begin{array}
[c]{cl}%
\mathrm{{I}_{4}^{\ast}} & \text{if }v_{2}(a_{1})=0,\\
\mathrm{{II}^{\ast}} & \text{if }v_{2}(a_{1})\geq1.
\end{array}
\right.
\]
Moreover,%
\begin{equation}
c_{2}(E^{D}/\mathbb{Q})=\left\{
\begin{array}
[c]{cl}%
1 & \text{if }v_{2}(a_{1})\geq1,\\
2 & \text{if }v_{2}(a_{1})=0\text{ and }a_{6}\equiv1,2\ \operatorname{mod}4,\\
4 & \text{if }v_{2}(a_{1})=0\text{ and }a_{6}\equiv0,3\ \operatorname{mod}4.
\end{array}
\right.  \label{Dval2Tama}%
\end{equation}
By (\ref{cptil1}), (\ref{Dval2Tama}), and Lemma \ref{lem:uvalue} we have that%
\begin{align*}
v_{2}(\Lambda) & =v_{2}\!\left(  c_{2}(E^{D}/\mathbb{Q})\right)
+\omega(N_{-})-\sum_{p|N_{-}}(v_{p}(\Delta)\ \operatorname{mod}2)-
\omega(N_{-})  \\
& \equiv\left\{
\begin{array}
[c]{cl}%
b\ \operatorname{mod}2 & \text{if\ }\left(  i\right)  \ v_{2}(a_{1})=0\text{
and }a_{6}\equiv0,3\ \operatorname{mod}4\text{ or }\left(  ii\right)
\ v_{2}(a_{1})\geq1\\
1+b\ \operatorname{mod}2 & \text{if }v_{2}(a_{1})=0\text{ and }a_{6}%
\equiv1,2\ \operatorname{mod}4.
\end{array}
\right.
\end{align*}
Thus,
\begin{equation}%
\begin{array}
[c]{l}%
\Lambda\text{ is a square}\qquad\Longleftrightarrow\qquad\\
\left\{
\begin{array}
[c]{ll}%
\left(  -1\right)  ^{b}=1 & \text{if\ }\left(  i\right)  \ v_{2}%
(a_{1})=0\text{ and }a_{6}\equiv0,3\ \operatorname{mod}4\text{ or }\left(
ii\right)  \ v_{2}(a_{1})\geq1\\
\left(  -1\right)  ^{b+1}=1 & \text{if }v_{2}(a_{1})=0\text{ and }a_{6}%
\equiv1,2\ \operatorname{mod}4.
\end{array}
\right.
\end{array}
\label{uvalueDval2}%
\end{equation}
In addition, Lemmas \ref{Lem:TamRecRel} and \ref{Lem:QuadRec} imply that
$\prod_{l|m}c_{l}(E^{D}/\mathbb{Q})$ is a power of $2$, and
\[
\prod_{l|m}c_{l}(E^{D}/\mathbb{Q})\text{ is a square}\qquad
\Longleftrightarrow\qquad\left(  \frac{\Delta}{m}\right)  =1
\]
where
\begin{equation}
\left(  \frac{\Delta}{m}\right)  =\left\{
\begin{array}
[c]{cl}%
\left(  -1\right)  ^{b} & \text{if }\Delta\equiv1\ \operatorname{mod}4,\\
\left(  -1\right)  ^{b+1} & \text{if }\Delta\equiv3\ \operatorname{mod}4.
\end{array}
\right.  \label{DelonmCase1}%
\end{equation}

\bigskip

\qquad\textbf{Subcase 1a.} Suppose $v_{2}(a_{1})=0$. Then, $v_{2}(a_{3})\geq2$ and either $\left(  i\right)  $ $v_{2}(a_{4})\geq1$ and $v_{2}(a_{6})=0$ or $\left(  ii\right)  \ v_{2}(a_{4})=0$ and $v_{2}(a_{6})\geq1$. Consequently,
\begin{align*}
\Delta &  \equiv a_{4}^{2}-a_{6}\ \operatorname{mod}4\\
&  \equiv\left\{
\begin{array}
[c]{cl}%
1\ \operatorname{mod}4 & \text{if }a_{6}\equiv0,3\ \operatorname{mod}4,\\
3\ \operatorname{mod}4 & \text{if }a_{6}\equiv1,2\ \operatorname{mod}4.
\end{array}
\right.
\end{align*}
From (\ref{DelonmCase1}),
\[
\left(  \frac{\Delta}{m}\right)  =\left\{
\begin{array}
[c]{cl}%
\left(  -1\right)  ^{b} & \text{if }a_{6}\equiv0,3\ \operatorname{mod}4,\\
\left(  -1\right)  ^{b+1} & \text{if }a_{6}\equiv1,2\ \operatorname{mod}4.
\end{array}
\right.
\]
Thus, if $a_{6}\equiv0,3\ \operatorname{mod}4$, then $\prod_{l|m}c_{l}(E^{D}/\mathbb{Q})$ is a square if and only if $b$ is even. But (\ref{uvalueDval2}) gives that $\Lambda$ is also a square if and only if $b$ is even, which shows that the theorem holds in this case. Similarly, if $a_{6}\equiv1,2\ \operatorname{mod}4$, then $\prod_{l|m}c_{l}(E^{D}/\mathbb{Q})$ is a square if and only if $b$ is odd. We also know from (\ref{uvalueDval2}) that $\Lambda$ is a square if and only if $b$ is odd, which concludes this subcase.

\bigskip

\qquad\textbf{Subcase 1b.} Suppose $v_{2}(a_{1})\geq1$ so that $v_{2}(a_{3})=0$. Then, $\Delta\equiv a_{3}^{4}\ \operatorname{mod}4$ implies that $\Delta\equiv1\ \operatorname{mod}4$. By Lemma \ref{Lem:QuadRec},
\[
\left(  \frac{\Delta}{m}\right)  =\left(  -1\right)  ^{b}.
\]
Consequently, Lemma \ref{Lem:TamRecRel} implies that $\prod_{l|m}c_{l}(E^{D}/\mathbb{Q})$ is a square if and only if $b$ is even. The theorem now follows in this case since (\ref{uvalueDval2}) gives that $\Lambda$ is a square if and only if $b$ is even.

\bigskip

\textbf{Case 2.} Suppose $v_{2}(D)=3$ and $v_{2}(c_{6})=3$. Thus, $v_{2}(a_{1})\geq1$ by (\ref{c6val}). Further, Lemma~\ref{Lem:QuadTwist} gives that $E^{D}$ and $E^{2m}$ are $\mathbb{Q}$-isomorphic. By \cite[Theorem~5.1]{BRSTTW},
\[
\operatorname*{typ}(E^{D})=\operatorname*{typ}(E^{2m})=\mathrm{{II}}
\]
and $c_{2}(E^{D}/\mathbb{Q})=1$. By (\ref{cptil1}) and Lemma \ref{lem:uvalue},
\begin{align*}
v_{2}(\Lambda) & =1+\omega(N_{-})-\sum_{p|N_{-}}(v_{p}(\Delta)\ \operatorname{mod}2)-  \omega(N_{-})  \\
& \equiv b+1\ \operatorname{mod}2.
\end{align*}
Therefore, $\Lambda$ is a square if and only if $\left(  -1\right)  ^{b+1}=1$. Since $v_{2}(a_{1})\geq1$, we have that $v_{2}(a_{3})=0$. These assumptions imply that $\Delta\equiv5a_{3}^{4}\ \operatorname{mod}8=5\ \operatorname{mod} 8$. By Lemma \ref{Lem:QuadRec}, $\left(  \frac{\Delta}{m}\right)  =\left(-1\right)  ^{b+1}$. It follows from Lemma \ref{Lem:TamRecRel} that
\[
\prod_{l|m}c_{l}(E^{D}/\mathbb{Q})
\]
is a square if and only if $b$ is odd. The theorem now follows in this case since $\Lambda$ is a square if and only if $b$ is odd.

\bigskip

\textbf{Case 3. }Suppose $v_{2}(D)=3$ and $v_{2}(c_{6})=0$. By (\ref{c6val}), $v_{2}(a_{1})=0$, $v_{2}(a_{3})\geq2$, and either $\left(  i\right)  $ $v_{2}(a_{4})=0$ and $v_{2}(a_{6})\geq1$ or $\left(  ii\right)  $ $v_{2}(a_{4})\geq1$ and $v_{2}(a_{6})=0$. By Lemma \ref{Lem:QuadTwist}, $E^{D}$ and $E^{2m}$ are $\mathbb{Q}$-isomorphic. By \cite[Theorem 5.1]{BRSTTW},
\[
\operatorname*{typ}(E^{D})=\operatorname*{typ}(E^{2m})=\mathrm{{I}_{8}^{\ast}.}
\]
Next, let%
\begin{align*}
P_{1} &  =4+16a_{2}+8a_{4}+4a_{6}-2m-2ma_{6}^{2}-4ma_{6},\\
P_{2} &  =a_{3}^{2}-2ma_{6}^{2}+4a_{6}.
\end{align*}
Then, {\it loc. cit.} implies that
\[
c_{2}(E^{D}/\mathbb{Q})=\left\{
\begin{array}
[c]{cl}%
2 & \text{if }\left(  i\right)  \ v_{2}(a_{6})=0\text{ and }v_{2}(P_{1})=4\text{ or }\left(  ii\right)  \ v_{2}(a_{6})\geq1\text{ and } v_{2}(P_{2})=4,\\
4 & \text{if }\left(  i\right)  \ v_{2}(a_{6})=0\text{ and }v_{2}(P_{1})\geq5\text{ or }\left(  ii\right)  \ v_{2}(a_{6})\geq1\text{ and }v_{2}(P_{2})\geq5.
\end{array}
\right.
\]
By (\ref{cptil1}) and Lemma \ref{lem:uvalue},
\begin{align*}
v_{2}(\Lambda) & =\omega(N_{-})-\sum_{p|N_{-}}\left(  v_{p}(\Delta)\ \operatorname{mod}2\right)  -  \omega(N_{-})  +v_{2}\!\left(  c_{2}(E^{D}/\mathbb{Q})\right)  \\
& \equiv\left\{
\begin{array}
[c]{cl}
b+1\ \operatorname{mod}2 & \text{if }\left(  i\right)  \ v_{2}(a_{6})=0\text{ and }v_{2}(P_{1})=4\text{ or }\left(  ii\right)  \ v_{2}(a_{6})\geq1\text{ and }v_{2}(P_{2})=4,\\
b\ \operatorname{mod}2 & \text{if }\left(  i\right)  \ v_{2}(a_{6})=0\text{ and }v_{2}(P_{1})\geq5\text{ or }\left(  ii\right)  \ v_{2}(a_{6})\geq1\text{ and }v_{2}(P_{2})\geq5.
\end{array}
\right.
\end{align*}
Hence,
{\small \begin{equation}%
\begin{array}
[c]{l}%
\Lambda\text{ is a square}\qquad\Longleftrightarrow\qquad\\
\left\{
\begin{array}
[c]{cl}%
\left(  -1\right)  ^{b+1}=1 & \text{if }\left(  i\right)  \ v_{2}%
(a_{6})=0\text{ and }v_{2}(P_{1})=4\text{ or }\left(  ii\right)  \ v_{2}%
(a_{6})\geq1\text{ and }v_{2}(P_{2})=4,\\
\left(  -1\right)  ^{b}=1 & \text{if }\left(  i\right)  \ v_{2}(a_{6})=0\text{
and }v_{2}(P_{1})\geq5\text{ or }\left(  ii\right)  \ v_{2}(a_{6})\geq1\text{
and }v_{2}(P_{2})\geq5.
\end{array}
\right.
\end{array}
\label{uvalueDval3}%
\end{equation}}

\bigskip

Now set
\begin{align*}
\mathcal{A}_{1} &  =\left\{  \left(  x_{1},x_{2},x_{3},x_{4},x_{6},y\right) \in\mathbb{Z}^{6}\mid v_{2}(x_{1})=v_{2}(x_{6})=v_{2}(y)=0,v_{2}(x_{3})\geq2,v_{2}(x_{4})\geq1\right\}  ,\\
\mathcal{A}_{2} &  =\left\{  \left(  x_{1},x_{2},x_{3},x_{4},x_{6},y\right) \in\mathbb{Z}^{6}\mid v_{2}(x_{1})=v_{2}(x_{4})=v_{2}(y)=0,v_{2}(x_{3})\geq2,v_{2}(x_{6})\geq1\right\}  .
\end{align*}
Thus, $\left(  a_{1},a_{2},a_{3},a_{4},a_{6},m\right)  \in\mathcal{A}=\mathcal{A}_{1}\cup\mathcal{A}_{2}$ and $\mathcal{A}_{1}\cap\mathcal{A}_{2}=\varnothing$. Now consider the natural projection $\pi:\mathcal{A}\rightarrow\left(  \mathbb{Z}/32\mathbb{Z}\right)  ^{6}$. By construction, $\pi(\mathcal{A}_{1})\cap\pi(\mathcal{A}_{2})=\varnothing$. For $X=\left(x_{1},x_{2},x_{3},x_{4},x_{6},y\right)  \in\pi(\mathcal{A})$, define $\tau :\pi(\mathcal{A})\rightarrow\mathbb{Z}/32\mathbb{Z}$ by
\[
\tau\left(  X\right)  =\left\{
\begin{array}
[c]{cl}
4+16x_{2}+8x_{4}+4x_{6}-2y-2yx_{6}^{2}-4yx_{6}\ \operatorname{mod}32 &
\text{if }X\in\pi(\mathcal{A}_{1}),\\
x_{3}^{2}-2yx_{6}^{2}+4x_{6}\ \operatorname{mod}32 & \text{if }X\in
\pi(\mathcal{A}_{2}).
\end{array}
\right.
\]
Then, $\tau\circ\pi:\mathcal{A}\rightarrow\mathbb{Z}/32\mathbb{Z}$ and by the proof of \cite[Theorem 5.1, $v(d)=1$, Subcases 1a,1b]{BRSTTW}, $\left(\tau\circ\pi\right)  (\mathcal{A})\in\left\{  0,16\right\}  $. Now set
\[
\mathcal{C}_{2}=\tau^{-1}(16)\qquad\text{and}\qquad\mathcal{C}_{4}=\tau^{-1}(0).
\]
Hence, $\pi(\mathcal{A})=\mathcal{C}_{2}\cup\mathcal{C}_{4}$. Moreover, for $\left(  a_{1},a_{2},a_{3},a_{4},a_{6},m\right)  \in\mathcal{A}$ we have that
\[
c_{2}(E^{D}/\mathbb{Q})=\left\{
\begin{array}
[c]{cl}%
2 & \text{if }\pi\!\left(  a_{1},a_{2},a_{3},a_{4},a_{6},m\right)
\in\mathcal{C}_{2},\\
4 & \text{if }\pi\!\left(  a_{1},a_{2},a_{3},a_{4},a_{6},m\right)
\in\mathcal{C}_{4}.
\end{array}
\right.
\]
Consequently, (\ref{uvalueDval3}) can be restated as
\begin{equation}
\Lambda\text{ is a square}\qquad\Longleftrightarrow\qquad\left\{
\begin{array}
[c]{cl}%
\left(  -1\right)  ^{b+1}=1 & \text{if }\pi\!\left(  a_{1},a_{2},a_{3}%
,a_{4},a_{6},m\right)  \in\mathcal{C}_{2},\\
\left(  -1\right)  ^{b}=1 & \text{if }\pi\!\left(  a_{1},a_{2},a_{3}%
,a_{4},a_{6},m\right)  \in\mathcal{C}_{4}.
\end{array}
\right.  \nonumber
\end{equation}
By Lemma \ref{Lem:TamRecRel}, the theorem has been reduced to proving that
\[
\left(  \frac{\Delta}{m}\right)  =\left\{
\begin{array}
[c]{cl}%
\left(  -1\right)  ^{b+1} & \text{if }\pi\!\left(  a_{1},a_{2},a_{3}%
,a_{4},a_{6},m\right)  \in\mathcal{C}_{2},\\
\left(  -1\right)  ^{b} & \text{if }\pi\!\left(  a_{1},a_{2},a_{3},a_{4}%
,a_{6},m\right)  \in\mathcal{C}_{4}.
\end{array}
\right.
\]
To this end, we note that
\begin{align*}
\Delta &  \equiv a_{1}^{5}a_{3}a_{4}-a_{1}^{6}a_{6}+a_{1}^{4}a_{4}^{2}+4a_{1}^{4}a_{2}a_{6}\ \operatorname{mod}8\\
&  \equiv a_{3}a_{4}-a_{6}+a_{4}^{2}+4a_{2}a_{6}\ \operatorname{mod}8\\
&  \equiv\left\{
\begin{array}
[c]{cl}
a_{4}^{2}+4a_{2}-a_{6}\ \operatorname{mod}8 & \text{if }\pi\!\left(
a_{1},a_{2},a_{3},a_{4},a_{6},m\right)  \in\pi(\mathcal{A}_{1}),\\
a_{3}-a_{6}+1\ \operatorname{mod}8 & \text{if }\pi\!\left(  a_{1},a_{2}
,a_{3},a_{4},a_{6},m\right)  \in\pi(\mathcal{A}_{2}).
\end{array}
\right.
\end{align*}
Now consider the functions $\lambda:\pi(\mathcal{A})\rightarrow\mathbb{Z}/8\mathbb{Z}$ and $\mu:\pi(\mathcal{A})\rightarrow\mathbb{Z}/4\mathbb{Z}$ defined by
\begin{align*}
\lambda\left(  x_{1},x_{2},x_{3},x_{4},x_{6},y\right)   &  =\left\{
\begin{array}
[c]{cl}%
x_{4}^{2}+4x_{2}-x_{6}\ \operatorname{mod}8 & \text{if }\left(  x_{1}%
,x_{2},x_{3},x_{4},x_{6},y\right)  \in\pi(\mathcal{A}_{1}),\\
x_{3}-x_{6}+1\ \operatorname{mod}8 & \text{if }\left(  x_{1},x_{2},x_{3}%
,x_{4},x_{6},y\right)  \in\pi(\mathcal{A}_{2}),
\end{array}
\right.  \\
\mu\left(  x_{1},x_{2},x_{3},x_{4},x_{6},y\right)   &  =y\ \operatorname{mod}
4.
\end{align*}
In particular, we have the following diagram:
\[
\begin{tikzcd}
& (\mathbb{Z}/32\mathbb{Z})^6                                                                & \mathbb{Z}/32\mathbb{Z} \\
\mathcal{A} \arrow[r, "\pi", two heads] & \pi(\mathcal{A}) \arrow[ru, "\tau"] \arrow[r, "\lambda"] \arrow[rd, "\mu"] \arrow[u, hook] & \mathbb{Z}/8\mathbb{Z}  \\
&                                                                                            & \mathbb{Z}/4\mathbb{Z}
\end{tikzcd}
\]
Observe that by construction
\begin{align*}
\left(  \lambda\circ\pi\right)  \left(  a_{1},a_{2},a_{3},a_{4},a_{6},m\right)   &  \equiv\Delta\ \operatorname{mod}8,\\
\left(  \mu\circ\pi\right)  \left(  a_{1},a_{2},a_{3},a_{4},a_{6},m\right)
&  \equiv m\ \operatorname{mod}4.
\end{align*}
In particular, if $\pi\!\left(  a_{1},a_{2},a_{3},a_{4},a_{6},m\right)  =X$,
then $\left(  \lambda(X),\mu(X)\right)  =\left(  \Delta\ \operatorname{mod}%
8,m\ \operatorname{mod}4\right)  $. In~\cite{GitHub}, it is shown that%
\[
\left\{  \left(  \lambda(X),\mu(X)\right)  \in\mathbb{Z}/8\mathbb{Z}%
\times\mathbb{Z}/4\mathbb{Z}\mid X\in\mathcal{C}_{j}\right\}  =\left\{
\begin{array}
[c]{cl}%
\left\{  \left(  3,1\right)  ,\left(  5,1\right)  ,\left(  5,3\right)
,\left(  7,3\right)  \right\}   & \text{if }j=2,\\
\left\{  \left(  1,1\right)  ,\left(  1,3\right)  ,\left(  3,3\right)
,\left(  7,1\right)  \right\}   & \text{if }j=4.
\end{array}
\right.
\]
By Lemma \ref{Lem:QuadRec}, we have that
\[
\left(  \frac{\Delta}{m}\right)  =\left\{
\begin{array}
[c]{cl}%
\left(  -1\right)  ^{b+1} & \text{if }\pi\!\left(  a_{1},a_{2},a_{3}%
,a_{4},a_{6},m\right)  \in\mathcal{C}_{2},\\
\left(  -1\right)  ^{b} & \text{if }\pi\!\left(  a_{1},a_{2},a_{3},a_{4}%
,a_{6},m\right)  \in\mathcal{C}_{4},
\end{array}
\right.
\]
which concludes the proof.
\end{proof}

\section{The case where $N_-$ has an even number of prime factors}\label{Sec5}

In this section, we consider the analogue of Theorem~\ref{1.1} in the other situation where $N_-$ is a squarefree product of an even number of primes. 

\bigskip

Thus, we consider the situation of an elliptic curve $E/\mathbb{Q}$, whose conductor $N$ is factored in the form $N=N_+ N_-$, where $N_+,N_-$ are relatively prime, and this time $N_-$ is a squarefree product of an {\it even} number of distinct primes (including the case $N_-=1$). Again consider positive fundamental discriminants $D$ that satisfy the modified Heegner hypothesis with respect to $(N_+,N_-)$: all primes dividing $N_+$ split in $\mathbb{Q}(\sqrt{D})$, and all primes dividing $N_-$ are inert in $\mathbb{Q}(\sqrt{D})$; of course when $N_-=1$, then the condition is just that all primes dividing $N_+=N$ split in $\mathbb{Q}(\sqrt{D})$. As before, $E^D$ is the quadratic twist of $E$ by $D$. In this case we have $\epsilon(E/\mathbb{Q}) =\epsilon(E^D/\mathbb{Q})$. Hence, the analytic ranks of $E/\mathbb{Q}$ and $E^D/\mathbb{Q}$ have the same parity. We are interested in the case where the analytic ranks are both equal to zero, and the case where the analytic ranks are both equal to one.

\bigskip

As in Section~\ref{Sec3}, we actually consider a more general situation. Thus, as in Section~\ref{Sec3}, we consider quadratic (or trivial) even primitive Dirichlet characters $\chi_1,\chi_2$, not both trivial, whose conductors are denoted as $D_1,D_2$, and put $D:=D_1 \cdot D_2$. As usual we assume that $D$ satisfies the modified Heegner hypothesis with respect to $(N_+,N_-)$: $\chi_1(l)=\chi_2(l)$ for all primes $l|N_+$, and $\chi_1(q)=-\chi_2(q)$ for all primes $q|N_-$. Then $\epsilon(E/\mathbb{Q},\chi_1) = \epsilon(E/\mathbb{Q},\chi_2)$, and the analytic ranks of both $E^{D_1}/\mathbb{Q}$ and $E^{D_2}/\mathbb{Q}$ have the same parity. 

\bigskip

Finally condition (*) of Section~\ref{Sec3} is again assumed to hold for the pair $(\chi_1,\chi_2)$. 

\bigskip

\begin{theorem}\label{5.1}
With notations as above, suppose that the analytic ranks of $E^{D_1}/\mathbb{Q}$ and $E^{D_2}/\mathbb{Q}$ are both equal to zero. Then the Birch and Swinnerton-Dyer formula modulo square of rational numbers holds for $E^{D_1}/\mathbb{Q}$ if and only if it holds for $E^{D_2}/\mathbb{Q} $. 
\end{theorem}

\bigskip

\begin{proof}
Thus we assume that $L(1,E^{D_1}/\mathbb{Q}) $ and $L(1,E^{D_2}/\mathbb{Q})$ are both non-zero, in particular $\epsilon(E/\mathbb{Q},\chi_1)=\epsilon(E/\mathbb{Q},\chi_2)=1$. By symmetry, it suffices to show that, assuming that the Birch and Swinnerton-Dyer formula modulo square of rational numbers is valid for $E^{D_1}/\mathbb{Q}$, i.e., we have:

\begin{eqnarray}\label{e5.1}
L(1,E^{D_1}/\mathbb{Q})  = \prod_{l | N D_1} c_l(E^{D_1}/\mathbb{Q}) \cdot \Omega^+_{E^{D_1}/\mathbb{Q}}
\mod{(\mathbb{Q}^{\times})^2},
\end{eqnarray}

\bigskip
\noindent then we need to show:

\begin{eqnarray}\label{e5.2}
L(1,E^{D_2}/\mathbb{Q})  = \prod_{l | N D_2} c_l(E^{D_2}/\mathbb{Q}) \cdot \Omega^+_{E^{D_2}/\mathbb{Q}}
\mod{(\mathbb{Q}^{\times})^2}
\end{eqnarray}

\bigskip

Let us first consider the case where $N_-=1$. As usual put $F=\mathbb{Q}(\sqrt{D})$, and $\chi_F$ the genus class character of $F$ corresponding to the pair $(\chi_1,\chi_2)$. Then we have the factorization of $L$-function:
\[
L(s,E/F,\chi_F) = L(s,E^{D_1}/\mathbb{Q}) \cdot L(s,E^{D_2}/\mathbb{Q})
\]
so in particular
\begin{eqnarray}\label{e5.3}
L(1,E/F,\chi_F) = L(1,E^{D_1}/\mathbb{Q}) \cdot L(1,E^{D_2}/\mathbb{Q}) \neq 0
\end{eqnarray}

\bigskip

Then in place of Theorem~\ref{2.1}, we use Popa's formula, namely Theorem 5.4.1 of \cite{Popa} with $N_-=1$, which tells us that:

\begin{eqnarray}\label{e5.4}
L(1,E/F,\chi_F) = \frac{1}{\sqrt{D}} (\Omega^+_{E/\mathbb{Q}} )^2 \mod{(\mathbb{Q}^{\times})^2} 
\end{eqnarray}
(here we remark that we are using Popa's formula in the classical setting as given in section 6 of {\it loc. cit.})

\bigskip

Then by Theorem~\ref{3.1}, in this particular case $N_-=1$, we have:

\begin{eqnarray}\label{e5.5}
u_{D_1} u_{D_2}  \prod_{l | D}c_l(E^{D_1}/\mathbb{Q})    \prod_{l | D}c_l(E^{D_1}/\mathbb{Q})  = 1 \mod{(\mathbb{Q}^{\times})^2} 
\end{eqnarray}

\bigskip
Thus combining \eqref{e5.1}, \eqref{e5.3} , \eqref{e5.4} and \eqref{e5.5} (and the definition of $u_{D_i}$, $i=1,2$), we see that \eqref{e5.2} is valid.

\bigskip

Next, we consider the case where $N_- \neq 1$. Pick any prime $p |N_-$, and put:
\[
N_+^{\prime} := N/p, \,\ N_-^{\prime} :=  p
\]
\[
N_+^{\prime \prime} := p N_+, \,\ N_-^{\prime \prime} :=  N_-/p
\]

\bigskip
Then $N=N_+^{\prime} N_-^{\prime}$ (resp. $N=N_+^{\prime \prime} N_-^{\prime \prime}$), with $N_+^{\prime}$ and $N_-^{\prime}$ being relatively prime (resp.  $N_+^{\prime \prime}$ and $N_-^{\prime \prime}$ being relatively prime), and $N_-^{\prime}$ (resp. $N_-^{\prime \prime}$) is a squarefree product of an {\it odd} number of distinct primes. 

\bigskip
We now choose an auxiliary quadratic primitive Dirichlet character $\chi_3$ satisfying the following conditions:

\begin{itemize}
\item (a) The conductor of $\chi_3$ is relatively prime to $N D$.
\item (b) $\chi_3$ is even, i.e. $\chi_3(-1)=1$.
\item (c) $\chi_3(l) = \chi_1(l)$ for all primes $l \neq p$ dividing $N$.
\item (d) $\chi_3(p) = -\chi_1(p)$. 
\item (e) $L^{\prime}(1,E/\mathbb{Q},\chi_3) \neq 0$.
\end{itemize}

\bigskip
The existence of such $\chi_3$ is again guaranteed by \cite{FH}. Indeed for any $\chi_3$ that satisfy conditions (a)-(d), we have that $\epsilon(E/\mathbb{Q},\chi_3) = -\epsilon(E/\mathbb{Q},\chi_1)=-1 $. And so by \cite{FH} we can choose $\chi_3$ satisfying (a)-(d) and such that $L^{\prime}(1,E/\mathbb{Q},\chi_3) \neq 0$. 

\bigskip

Fix such a $\chi_3$, and denote by $D_3$ the conductor of $\chi_3$. Thus $D_3$ is relatively prime to $N D = N D_1 D_2$. The analytic rank of $E^{D_3}/\mathbb{Q}$ is thus equal to one.

\bigskip
Put $D^{\prime} :=D_1 D_3$ (corresponding to the primitive Dirichlet character $\chi_1 \cdot \chi_3$) and $D^{\prime \prime} :=D_2 D_3$ (corresponding to the primitive Dirichlet character $\chi_2 \cdot \chi_3$). 

\bigskip
Now we note that $D^{\prime} =D_1 D_3$ satisfies the modified Heegner hypothesis with respect to $(N_+^{\prime},N_-^{\prime})$, namely that $\chi_1(l)=\chi_3(l)$ for all primes $l|N_+^{\prime}$, and $\chi_1(q)=-\chi_3(q)$ for all primes $q|N_-^{\prime}$ (in this case there is only one such $q$, namely $p$); also it is clear that condition (*) of Section~\ref{Sec3} also holds for the pair $(\chi_1,\chi_3)$. So we can apply Theorem~\ref{3.4} to the pair $(D_1,D_3)$; in particular from the validity of the Birch and Swinnerton-Dyer formula modulo square of rational numbers for $E^{D_1}/\mathbb{Q}$, we obtain the validity of the Birch and Swinnerton-Dyer formula modulo square of rational numbers for $E^{D_3}/\mathbb{Q}$.

\bigskip

In turn we then note that $D^{\prime \prime} =D_2 D_3$ satisfies the modified Heegner hypothesis with respect to $(N_+^{\prime \prime},N_-^{\prime \prime})$, namely that $\chi_2(l)=\chi_3(l)$ for all primes $l|N_+^{\prime \prime}$, and $\chi_2(q)=-\chi_3(q)$ for all primes $q|N_-^{\prime \prime}$; also condition (*) of Section~\ref{Sec3} again holds for the pair $(\chi_2,\chi_3)$. So we can again apply Theorem~\ref{3.4}, this time to the pair $(D_2,D_3)$; in particular, we obtain the validity of the Birch and Swinnerton-Dyer formula modulo square of rational numbers for $E^{D_2}/\mathbb{Q}$ from that of $E^{D_3}/\mathbb{Q}$.

\bigskip
This finishes the proof of Theorem~\ref{5.1}

\end{proof}

\bigskip

Finally we consider the case where both $E^{D_1}/\mathbb{Q}$ and $E^{D_2}/\mathbb{Q}$ have analytic rank one, in particular $\epsilon(E/\mathbb{Q},\chi_1) = \epsilon(E/\mathbb{Q},\chi_2)=-1$. For the proof of Theorem~\ref{5.2} we need to assume that $E/\mathbb{Q}$ has at least one prime of multiplicative reduction.

\bigskip

\begin{theorem}\label{5.2}
Assume that $E/\mathbb{Q}$ has at least one prime of multiplicative reduction. Suppose that the analytic ranks of $E^{D_1}/\mathbb{Q}$ and $E^{D_2}/\mathbb{Q}$ are both equal to one. Then the Birch and Swinnerton-Dyer formula modulo square of rational numbers holds for $E^{D_1}/\mathbb{Q}$ if and only if it holds for $E^{D_2}/\mathbb{Q} $. 
\end{theorem}
\begin{proof}
We first consider the case where $N_- \neq 1$. Then the argument is very much similar to that of Theorem~\ref{5.1}: pick any prime $p |N_-$, and put:
\[
N_+^{\prime} := N/p, \,\ N_-^{\prime} :=  p
\]
\[
N_+^{\prime \prime} := p N_+, \,\ N_-^{\prime \prime} :=  N_-/p
\]

Then $N=N_+^{\prime} N_-^{\prime}$ (resp. $N=N_+^{\prime \prime} N_-^{\prime \prime}$), with $N_+^{\prime}$ and $N_-^{\prime}$ being relatively prime (resp.  $N_+^{\prime \prime}$ and $N_-^{\prime \prime}$ being relatively prime), and $N_-^{\prime}$ (resp. $N_-^{\prime \prime}$) is a squarefree product of an odd number of distinct primes. 

\bigskip
We now choose an auxiliary quadratic primitive Dirichlet character $\chi_3$ satisfying the following conditions:

\begin{itemize}
\item (a) The conductor of $\chi_3$ is relatively prime to $N D$.
\item (b) $\chi_3$ is even, i.e. $\chi_3(-1)=1$.
\item (c) $\chi_3(l) = \chi_1(l)$ for all primes $l \neq p$ dividing $N$.
\item (d) $\chi_3(p) = -\chi_1(p)$. 
\item (e) $L(1,E/\mathbb{Q},\chi_3) \neq 0$.
\end{itemize}

\bigskip
The existence of such $\chi_3$ is again guaranteed by \cite{FH}, or this time we can also use Murty-Murty \cite{MM}. Indeed for any $\chi_3$ that satisfy conditions (a)-(d), we have that $\epsilon(E/\mathbb{Q},\chi_3) = -\epsilon(E/\mathbb{Q},\chi_1)=+1 $. And so by \cite{FH} or \cite{MM} we can choose $\chi_3$ satisfying (a)-(d) and such that $L(1,E/\mathbb{Q},\chi_3) \neq 0$. 

\bigskip

Fix such a $\chi_3$, and denote by $D_3$ the conductor of $\chi_3$ (and $D_3$ is relatively prime to $N D = N D_1 D_2$). The analytic rank of $E^{D_3}/\mathbb{Q}$ is thus equal to zero. In a similar way put $D^{\prime} :=D_1 D_3$ (corresponding to the primitive Dirichlet character $\chi_1 \cdot \chi_3$) and $D^{\prime \prime} :=D_2 D_3$ (corresponding to the primitive Dirichlet character $\chi_2 \cdot \chi_3$). 

\bigskip
Just as before we have that $D^{\prime} =D_1 D_3$ satisfies the modified Heegner hypothesis with respect to $(N_+^{\prime},N_-^{\prime})$, and that condition (*) of Section~\ref{Sec3} also holds for the pair $(\chi_1,\chi_3)$. So we can apply Theorem~\ref{3.4} to the pair $(D_1,D_3)$; in particular from the validity of the Birch and Swinnerton-Dyer formula modulo square of rational numbers for $E^{D_1}/\mathbb{Q}$, we obtain the validity of the Birch and Swinnerton-Dyer formula modulo square of rational numbers for $E^{D_3}/\mathbb{Q}$.

\bigskip

And similarly just as before $D^{\prime \prime} =D_2 D_3$ satisfies the modified Heegner hypothesis with respect to $(N_+^{\prime \prime},N_-^{\prime \prime})$, and that condition (*) of Section~\ref{Sec3} holds for the pair $(\chi_2,\chi_3)$. So we can again apply Theorem~\ref{3.4}, to the pair $(D_2,D_3)$; we thus obtain the validity of the Birch and Swinnerton-Dyer formula modulo square of rational numbers for $E^{D_2}/\mathbb{Q}$ from that of $E^{D_3}/\mathbb{Q}$. This proves the theorem in the case when $N_- \neq 1$.

\bigskip

Finally, in the remaining case where $N_-=1$ (and so $N_+=N$), then by assumption $E/\mathbb{Q}$ has multiplicative reduction at some prime $p$ dividing $N_+=N$. Fix any such $p$, which then divides $N$ exactly (in particular, $p$ and $N/p$ are relatively prime). And put:
\[
N^{\prime}_+ :=N/p, \,\ N^{\prime}_-=p
\]

\bigskip

By \cite{FH} or \cite{MM} again, there exists a quadratic primitive Dirichlet character $\chi_3$ that satisfies the following:
\begin{itemize}
\item (a) The conductor of $\chi_3$ is relatively prime to $N D$.
\item (b) $\chi_3$ is even, i.e. $\chi_3(-1)=1$.
\item (c) $\chi_3(l) = \chi_1(l) =\chi_2(l)$ for all primes $l \neq p$ dividing $N$.
\item (d) $\chi_3(p) = -\chi_1(p) = -\chi_2(p)$. 
\item (e) $L(1,E/\mathbb{Q},\chi_3) \neq 0$.
\end{itemize}

\bigskip
Fix such a $\chi_3$, and denote by $D_3$ the conductor of $\chi_3$ (and $D_3$ is relatively prime to $N D = N D_1 D_2$). The analytic rank of $E^{D_3}/\mathbb{Q}$ is thus equal to zero. In a similar way put $D^{\prime} :=D_1 D_3$ (corresponding to the primitive Dirichlet character $\chi_1 \cdot \chi_3$) and $D^{\prime \prime} :=D_2 D_3$ (corresponding to the primitive Dirichlet character $\chi_2 \cdot \chi_3$). 

\bigskip
Note that $D^{\prime} =D_1 D_3$ satisfies the modified Heegner hypothesis with respect to $(N_+^{\prime},N_-^{\prime})$, and that condition (*) of Section~\ref{Sec3} holds for the pair $(\chi_1,\chi_3)$. Similarly $D^{\prime \prime} =D_2 D_3$ also satisfies the modified Heegner hypothesis with respect to $(N_+^{ \prime},N_-^{ \prime})$, and that condition (*) of Section~\ref{Sec3} holds for the pair $(\chi_2,\chi_3)$. Thus, we can conclude the proof by the same argument as before.

\end{proof}

\bigskip

Finally it is clear that Theorem~\ref{1.4} is a special case of Theorem~\ref{5.1} and Theorem~\ref{5.2}.

\bigskip

\noindent {\bf Remark 5.3.} For the proof of Theorem~\ref{5.1} in the case when $N_- =1$, we then see that, if $E/\mathbb{Q}$ has at least one prime of multiplicative reduction, then we could also establish the result by using an auxiliary character, and so the use of Popa's formula could then be avoided.

\bigskip

We now conclude the paper with the following theorem, which is the more general version of Corollary~\ref{1.5}:

\bigskip

\begin{theorem}\label{5.3}
Let $E/\mathbb{Q}$ be an elliptic curve with conductor $N$. Consider positive fundamental discriminants $D_1,D_2$ that are relatively prime to $N$, and $D_1,D_2$ being relatively prime, such that for all primes $l$ of additive reduction of $E/\mathbb{Q}$ we have that $l$ splits in both $\mathbb{Q}(\sqrt{D_1})$ and $\mathbb{Q}(\sqrt{D_2})$. Suppose that the analytic ranks of both $E^{D_1}/\mathbb{Q}$ and $E^{D_2}/\mathbb{Q}$ are at most one, and in the case where both $E^{D_1}/\mathbb{Q}$ and $E^{D_2}/\mathbb{Q}$ have analytic rank one, we assume in addition that $E/\mathbb{Q}$ has at least one prime of multiplicative reduction. Then we have that the Birch and Swinnerton-Dyer formula modulo square of rational numbers holds for $E^{D_1}/\mathbb{Q}$ if and only if it holds for $E^{D_2}/\mathbb{Q}$.
\end{theorem}
\begin{proof}
The proof is similar to that of Corollary~\ref{1.5}. Given $D_1,D_2$ as in the statement of Theorem~\ref{5.3}, let $\chi_1$ and $\chi_2$ be the Kronecker symbols corresponding to the positive fundamental discriminants $D_1$ and $D_2$, respectively. Define $N_+$ and $N_-$ as (here below $n_l$ is the exact power of a prime $l$ dividing $N$):
\[
N_+ := \prod_{l|N,\chi_1(l)=\chi_2(l)} l^{n_l}
\]
\[
N_- := \prod_{l|N, \chi_1(l)=-\chi_2(l)} l
\]
Then $N_+$ and $N_-$ are relatively prime with $N_-$ being squarefree, and by our assumption on $E/\mathbb{Q}$ at primes of additive reduction, we have $N=N_+ N_-$. Also, $D:=D_1 \cdot D_2$ satisfies the modified Heegner hypothesis with respect to $(N_+,N_-)$. In addition condition (*) of Section~\ref{Sec3} holds for the pair $(\chi_1,\chi_2)$, again by our assumption on $E/\mathbb{Q}$ at primes of additive reduction (namely that for primes $l$ of additive reduction we have $\chi_1(l)=\chi_2(l)=1$).

\bigskip
Assuming now that both the analytic ranks of $E^{D_1}/\mathbb{Q}$ and $E^{D_2}/\mathbb{Q}$ are at most one. If $N_-$ is a squarefree product of an odd number of distinct prime factors, then it must be the case that $E^{D_1}/\mathbb{Q}$ has analytic rank zero and $E^{D_2}/\mathbb{Q}$ has analytic rank one, or the case that $E^{D_1}/\mathbb{Q}$ has analytic rank one and $E^{D_2}/\mathbb{Q}$ has analytic rank zero, and so we apply Theorem~\ref{3.4}.  If $N_-$ is a squarefree product of an even number of distinct prime factors, then it must be the case that $E^{D_1}/\mathbb{Q}$ and $E^{D_2}/\mathbb{Q}$ both have analytic rank zero, for which we apply Theorem~\ref{5.1}, or the case that $E^{D_1}/\mathbb{Q}$ and $E^{D_2}/\mathbb{Q}$ both have analytic rank one, for which we apply Theorem~\ref{5.2} (here we are using the assumption that in this case $E/\mathbb{Q}$ at least one prime of multiplicative reduction for $E/\mathbb{Q}$, so that Theorem~\ref{5.2} is indeed applicable). 
\end{proof}

\bibliographystyle{plain}
\bibliography{bibliography}

\end{document}